\documentclass[11pt,oneside,reqno]{amsart}
\RequirePackage{fix-cm} 
\usepackage[T1]{fontenc}
\usepackage[english]{babel}
\usepackage{amssymb,amsmath,amsthm,graphicx,amsfonts}
\usepackage[dvipsnames]{xcolor}
\usepackage{lmodern}
\usepackage{authblk,tikz,cite,enumerate}

\numberwithin{equation}{section}

\usepackage[normalem]{ulem}

\usepackage{hyperref}
\hypersetup{
 colorlinks,
 citecolor=Maroon,
 linkcolor=Maroon,
 urlcolor=Maroon}

\usepackage[foot]{amsaddr}
 
\usepackage{titlesec}
\titleformat{\section}[block]{\large \bfseries\filcenter}{\thesection}{1em}{}
\titleformat{\subsection}[hang]{\bfseries}{\thesubsection}{1em}{}

\makeatletter
\def\@seccntformat#1{\hspace*{0mm}%
  \protect\textup{\protect\@secnumfont
    \ifnum\pdfstrcmp{subsection}{#1}=0 \bfseries\fi
    \csname the#1\endcsname
    \protect\@secnumpunct
  }%
}
\makeatother

\newtheorem{proposition}{Proposition}

\newtheorem{theorem}{Theorem}
\theoremstyle{remark}
\newtheorem{remark}{\sc  Remark\rm}[section]

\newcommand{\assign}{:=}
\newcommand{\mathd}{\mathrm{d}}

\newcommand{\tmabbr}[1]{#1}

\newcommand{\tmem}[1]{{\em #1\/}}
\newcommand{\tmmathbf}[1]{\ensuremath{\boldsymbol{#1}}}
\newcommand{\tmmathmd}[1]{\ensuremath{#1}}
\newcommand{\tmop}[1]{\ensuremath{\operatorname{#1}}}

\newcommand{\tmstrong}[1]{\textbf{#1}}

\newcommand{\vBonecurl}{\mathring{W}^1(\mathrm{curl},\mathbb{R}^3)}
\newcommand{\vWonecurl}{\mathring{H}^1(\mathrm{curl},\mathbb{R}^3)}
\newcommand{\vWonecurlh}{\mathring{H}_{\mathrm{sol}}^1(\ensuremath{\mathrm{curl}},\mathbb{R}^3
  )} \newcommand{\Bonegrad}{\mathring{W}^1(\mathbb{R}^3)}
\newcommand{\Wonegradh}{\mathring{H}^1(\mathbb{R}^3)}
\newcommand{\D}{\mathcal{D}(\mathbb{R}^3)}
\newcommand{\vD}{\mathcal{D}(\mathbb{R}^3;\mathbb{R}^3)}
\newcommand{\Ltwo}{L^2(\mathbb{R}^3)}
\newcommand{\vLtwo}{L^2(\mathbb{R}^3;\mathbb{R}^3)}
\newcommand{\Ltwow}{L^2_{\omega}(\mathbb{R}^3)}
\newcommand{\vLtwow}{L^2_{\omega}(\mathbb{R}^3;\mathbb{R}^3)}
\newcommand{\vLtwoww}{L^2_{\omega^{-1}}(\mathbb{R}^3;\mathbb{R}^3)}
\newcommand{\vDsol}{\mathcal{D}_{\mathrm{sol}}(\mathbb{R}^3;\mathbb{R}^3)}
\newcommand{\Dp}{\mathcal{D}'(\mathbb{R}^3)}
\newcommand{\vDp}{\mathcal{D}'(\mathbb{R}^3;\mathbb{R}^3)}

\newcommand{\hd}{\tmmathbf{h}}
\newcommand{\ha}{\tmmathbf{h}_a}
\newcommand{\bd}{\tmmathbf{b}}
\newcommand{\ud}{\tmmathbf{u}}
\newcommand{\ad}{\tmmathbf{a}}
\newcommand{\m}{\tmmathbf{m}}

\newcommand{\dual}[2]{\tmmathbf{\left\langle \tmmathmd{\mathrm{#1}}, \tmmathmd{\mathrm{#2}} \right\rangle}}

\newcommand{\RR}{\mathbb{R}}

\newcommand{\NN}{\mathbb{N}}

\newcommand{\R}{\mathbb{R}}
\newcommand{\Stwo}{\mathbb{S}}

\newcommand{\lapl}{\Delta}
\newcommand{\divv}{\mathrm{div\,}}

\newcommand{\grad}{\mathrm{{\nabla}}}
\newcommand{\curl}{\ensuremath{\mathrm{curl\,}}}

\newcommand{\st}{\mathbf{\; :\; }}
\newcommand{\eqs}{=}

\definecolor{ao(english)}{rgb}{0.0, 0.5, 0.0}

\DeclareRobustCommand{\subtitle}[1]{\\#1}

\begin{document}

\title[Variational principles of micromagnetics]{Variational
  principles of micromagnetics revisited}
        
\author{Giovanni Di Fratta}
\address{{Giovanni Di Fratta,\,\small{  Institute for Analysis and Scientific Computing, TU Wien\\
Wiedner Hauptstra{\ss}e 8-10 \\
1040 Wien, Austria.}}}

\author{Cyrill B. Muratov}
\address{{Cyrill B. Muratov,\,\small{  New Jersey Institute of Technology \\
Newark NJ 07102, USA.}}}

\author{Filipp N. Rybakov}
\address{{Filipp N. Rybakov,\,\small{  Department of Physics \\KTH-Royal Institute of Technology \\ Stockholm, SE-10691 Sweden}}}

\author{Valeriy V. Slastikov}
\address{{Valeriy V. Slastikov,\,\small{School of Mathematics \\
University of Bristol \\
Bristol BS8 1TW, United Kingdom.}}}

\begin{abstract}
  
  We revisit the basic variational formulation of the minimization
  problem associated with the micromagnetic energy, with an emphasis
  on the treatment of the stray field contribution to the energy,
  which is intrinsically non-local. Under minimal assumptions, we
  establish three distinct variational principles for the stray field
  energy: a minimax principle involving magnetic scalar potential and
  two minimization principles involving magnetic vector potential. We
  then apply our formulations to the dimension reduction problem for
  thin ferromagnetic shells of arbitrary shapes.
  
\vspace{11pt}
\noindent \tmstrong{Keywords.} Micromagnetics, Maxwell's equations,
stray field, minimizers, $\Gamma$-convergence

\end{abstract}

\maketitle
\renewcommand{\subtitle}[1]{}

\section{Introduction}

Ferromagnetism is a striking and subtle phenomenon. Observable on the
macroscopic scale, ferromagnetism has its origins from the two
quintessentially quantum mechanical properties of matter, namely the
electron spin and the Pauli exclusion principle \cite{aharoni}. The
quantum mechanical origin of ferromagnetism accounts for the existence
of a multitude of intriguing spin textures, from macroscopic down to
single nanometer scales
\cite{hubert,Arrott2016,FerPac2017,Donnelly2017}.
%
%
The small size of the magnetization patterns, along with the modest energy
required to manipulate them has produced and is continuing to lead to
far-reaching applications in information technology
\cite{piramanayagam07,bader10,apalkov16,bhatti17}.

There is a well established and extremely successful continuum theory
of micromagnetism, the {\it micromagnetic variational principle}, that
describes the equilibrium and dynamic magnetization configurations
\cite{brown,hubert,landau8,fidler00,oommf,leliaert18}. In this theory,
magnetization is described by a spatially varying vector field
$\mathbf M$, and stable magnetization configurations correspond to
global and local minimizers of the micromagnetic energy -- a
non-convex, nonlocal functional involving multiple length scales.  The
micromagnetic energy associated with the magnetization state of a
ferromagnetic sample occupying three-dimensional bounded domain
$\Omega$ ($\Omega \subset\R^3$) is
\cite{hubert,landau8,bertotti1998hysteresis}
\begin{align}
  \label{Ephys}
  E(\mathbf M) = {A \over M_\mathrm{s}^2} \int_\Omega |\nabla \mathbf M|^2 \,
  d^3 r + K \int_\Omega \Phi \left( \frac{\mathbf M}{M_\mathrm{s}} \right) d^3 r
  - \frac{\mu_0}{2} \int_\Omega \mathbf H_\mathrm{d} \cdot \mathbf M\, d^3 r -
  \mu_0 \int_\Omega \mathbf H_\mathrm{a} \cdot \mathbf M \, d^3 r,
\end{align}
where $\mathbf M = (M_1, M_2, M_3)$ is the magnetization vector that
satisfies $|\mathbf M|=M_\mathrm{s}$ in $\Omega$ and
$\mathbf M = \mathbf 0$ in $\R^3\setminus \Omega$ (i.e., outside the
domain $\Omega$), the positive constants $M_\mathrm{s}$, $A$ and $K$
are the saturation magnetization and exchange and anisotropy
constants, respectively, $\mathbf H_\mathrm{a}$ is the applied
magnetic field, and $\mu_0$ is the permeability of vacuum. Here we use
the standard notation
$|\nabla \mathbf M|^2 = |\nabla M_1|^2 + |\nabla M_2|^2 + |\nabla
M_3|^2$ for the Euclidean norm of gradients of vectorial quantities.
All physical quantities are assumed to be in the SI units.  The
demagnetizing field $\mathbf H_\mathrm{d}$ is determined via the
magnetic induction
$\mathbf B = \mathbf B_\mathrm{a} + \mathbf B_\mathrm{d}$, where
$\mathbf B_\mathrm{a} = \mu_0 \mathbf H_\mathrm{a}$ is the induction
in the absence of the ferromagnet due to permanent external field
sources, and
\begin{align}
  \label{B}
  \mathbf B_\mathrm{d} = \mu_0 (\mathbf H_\mathrm{d} + \mathbf M).
\end{align}
The pair $(\mathbf H_\mathrm{d}, \mathbf B_\mathrm{d})$ solves the
following system obtained from the time-independent Maxwell's
equations:
\begin{align} \label{eq:Maxwell} \divv \mathbf B_\mathrm{d} =0, \qquad
  \curl \mathbf H_\mathrm{d} = \mathbf 0,
\end{align}
where we noted that by definition $\divv \mathbf B_\mathrm{a} = 0$ in
$\mathbb R^3$. In \eqref{Ephys}, the terms in the order of appearance
are the exchange, $E_\mathrm{ex}$, magnetocrystalline anisotropy,
$E_\mathrm{a}$, stray field, $E_\mathrm{s}$, and Zeeman,
$E_\mathrm{Z}$, energies, respectively.

There exist several well-known representations of the stray field
energy employed in the analysis of the micromagnetic energy
\cite{brown0}. Using \eqref{eq:Maxwell}, one can introduce the magnetic scalar
potential $U_\mathrm{d}\,:\, \R^3 \to \R$ associated with the
demagnetizing field, such that
$\mathbf H_\mathrm{d} = -\nabla U_\mathrm{d}$, and $U_\mathrm{d}$
satisfies the following equation in the sense of distributions
\begin{align}
  \Delta U_\mathrm{d} = \divv \mathbf M
\end{align}
and vanishes at infinity. The stray field energy can be rewritten in
terms of $U_\mathrm{d}$ as \cite{brown0} 
\begin{align}
  \label{eq:EU}
  E_\mathrm{s}(\mathbf M) = \frac{\mu_0}{2}  \int_\Omega \mathbf M
  \cdot \nabla 
  U_\mathrm{d}\, d^3 r = \frac{\mu_0}{2} \int_{\R^3} |\nabla
  U_\mathrm{d}|^2 \, d^3 r.  
\end{align}
Using the fundamental solution of the Laplace equation in $\R^3$, one
can also rewrite the stray field energy in the following way
\begin{align}
  E_\mathrm{s} (\mathbf M) = {\mu_0 \over 8 \pi} \int_{\R^3} \int_{\R^3} {\divv
  \mathbf M(\mathbf r) \, \divv \mathbf M(\mathbf r') \over |
  \mathbf r - \mathbf r'|}  \, d^3 r \, d^3 r',  
\end{align}
reflecting its nonlocal and singular nature. Note that since
$\mathbf M$ has a jump at the boundary of domain $\Omega$, its
divergence $\divv \mathbf M$ has a singularity and, therefore must be
understood in a formal sense through its Fourier symbol.


Another way to represent the stray field energy is to employ the
magnetic vector potential $\mathbf A$ satisfying
$\mathbf B = \curl \mathbf A = \curl( \mathbf A_\mathrm{a} + \mathbf
A_\mathrm{d})$, where $ \mathbf A_\mathrm{a}$ and
$\mathbf A_\mathrm{d}$ are the contributions associated with
$\mathbf B_\mathrm{a}$ and $\mathbf B_\mathrm{d}$, respectively.
The magnetic vector potential is unobservable and not uniquely defined
due to gauge invariance. However, this potential is contained in the
momentum operator for a charged particle and therefore plays a crucial
role in the description of superconductivity and
Ehrenberg-Siday-Aharonov-Bohm effect underlying the method of electron
holography~\cite{Lichte_2007}.  In the Coulomb gauge one sets
$\divv \mathbf A_\mathrm{a} = \divv \mathbf A_\mathrm{d} =0$, leading
to the following equation for $\mathbf A_\mathrm{d}$ understood in the
sense of distributions \cite{brown0}:
\begin{align}
  \curl ( \curl \mathbf A_\mathrm{d}) = -\Delta \mathbf A_\mathrm{d} =
  \mu_0 \, \curl \mathbf M, 
\end{align}
where we used the identity
$\grad (\divv \mathbf A) - \curl (\curl \mathbf A) = \Delta \mathbf
A$.  In a similar way as with the use of magnetostatic potential
$U_\mathrm{d}$, we can rewrite the demagnetizing field
$\mathbf H_\mathrm{d} = \mu_0^{-1} \curl \mathbf A_\mathrm{d} -
\mathbf M$ to represent the stray field energy as
\begin{align}
  E_\mathrm{s}(\mathbf M) = \frac12 \int_\Omega  \left( \mu_0 |\mathbf
  M|^2 - \mathbf M \cdot \curl \mathbf A_\mathrm{d}
  \right) d^3 r =  \frac{1}{2\mu_0} \int_{\R^3} \left|
  \curl \mathbf A_\mathrm{d} - \mu_0 \mathbf M \right|^2  \, d^3 r. 
\end{align}
Again, using the fundamental solution of the Laplace equation in
$\R^3$ we obtain another representation of the stray field energy: 
\begin{align}
  \label{Escurlcurl}
  E_\mathrm{s}(\mathbf M) = \frac{1}{2} \mu_0 M_\mathrm{s}^2 |\Omega|
  - {\mu_0 \over 8 \pi} \int_{\R^3}\int_{\R^3} {\curl \mathbf
  M(\mathbf r) \cdot 
  \curl \mathbf M(\mathbf r') \over | \mathbf r - \mathbf r'|}
  \, d^3 r \, d^3 r',
\end{align}
where $|\Omega|$ is the volume of $\Omega$.  Note that since
$\mathbf M$ has a jump at the boundary of domain $\Omega$,
$\curl \mathbf M$ has a singularity and, therefore must again be
understood in a formal sense through its Fourier symbol.

The multi-scale complexity of the micromagnetic energy allows for a
variety of distinct regimes characterized by different relations
between material and geometrical parameters, and makes the
micromagnetic theory very rich and challenging
\cite{hubert,desimone06r}. One of the most powerful analytical
approaches to study the equilibria of the micromagnetic energy is the
investigation of its $\Gamma$-limits in various asymptotic regimes. To
achieve this, one needs to obtain asymptotically matching lower and
upper bounds for the micromagnetic energy. Typically, the construction
of the upper bounds is done using appropriate test functions; the
lower bound constructions are more difficult and require a careful
analysis of the specific problem under consideration. We point out,
however, that in the case of the stray field energy even constructing
the upper bounds might present a significant challenge due to the
non-local and singular behavior of the demagnetizing field
$\mathbf H_\mathrm{d}$.

In this paper, we revisit the variational formulation associated with
the micromagnetic energy, emphasizing the treatment of the stray field
energy to obtain efficient upper and lower bounds. To this aim, we
formulate three distinct variational principles for local minimizers
of the micromagnetic energy. The first variational principle can be
stated as a minimax problem for the magnetization $\mathbf M$ and the
scalar potential $U$. Specifically, for $\mathbf M$ fixed the stray
field energy may be expressed as
\begin{align}
  \label{eq:Es1}
  E_\mathrm{s}(\mathbf M) = \max_{U \in \mathring{H}^1(\mathbb R^3)}
  \, \mu_0  \int_{\mathbb R^3} \left( \mathbf M \cdot \nabla U -
  \frac12 |\nabla U|^2 \right) d^3 r 
\end{align}
and, therefore, yields convenient lower bounds on the stray field
energy via the use of test functions for $U$ (recall that
$\mathring{H}^1(\mathbb R^3)$ denotes the space of functions whose
first derivatives are square integrable; see section \ref{s:setup} for
the precise definitions of the function spaces).

The second variational principle is a joint minimization problem for
the magnetization $\mathbf M$ and the vector potential $\mathbf A$
subject to the Coulomb gauge ($\divv \mathbf A = 0$), with the stray
field energy expressed as
\begin{align}
  \label{eq:Es2}
  E_\mathrm{s}(\mathbf M) = \min_{\stackrel{\mathbf A \in 
  \mathring{H}^1(\mathbb R^3; \mathbb R^3)}{\divv \mathbf A = 0}} \, 
  \frac{1}{2\mu_0} \int_{\R^3} \left| 
  \curl \mathbf A - \mu_0 \mathbf M \right|^2  \, d^3 r
\end{align}
and is useful in constructing upper bounds for the stray field energy
via suitable test functions for $\mathbf A$.

Finally, we introduce the third variational principle closely linked
to the second one that amounts to a joint minimization for the
magnetization $\mathbf M$ and the vector potential $\mathbf A$ in the
absence of the constraint on $\divv \mathbf A$. It allows to express
the stray field energy in the form
\begin{align}
  \label{eq:Es3}
  E_\mathrm{s}(\mathbf M) = \frac{1}{2} \mu_0 M_\mathrm{s}^2 V +
  \min_{\mathbf A \in 
  \mathring{H}^1(\mathbb R^3; \mathbb R^3)}  \int_{\mathbb R^3} 
  \left(  \frac{1}{2 \mu_0} |\nabla \mathbf A|^2 - \mathbf M \cdot
  \curl \mathbf A \right) d^3 r.
\end{align}
This formula gives a novel representation of the magnetostatic energy,
which is particularly convenient both for obtaining localized upper
bounds for the micromagnetic energy and the numerical implementation
of the stray field.


The variational principle in \eqref{eq:Es1} leading to \eqref{eq:EU}
is well-known. In the context of micromagnetics, where one needs to
minimize the energy in \eqref{Ephys} with respect to $\mathbf M$ with
$\mathbf H_\mathrm{d}$ determined by the unique solution of
\eqref{eq:Maxwell}, it results in a minimax problem in terms of the
pair $(\mathbf M, U)$. As such, this minimax principle has not been
precisely formulated in the literature, although it has long existed
in the micromagnetics folklore (see, e.g., \cite{
  brown,brown0,james90}). Here we establish the validity of this
variational principle under minimal assumptions that arise naturally
in the context of micromagnetics.

Similarly, the minimization principles for the micromagnetic energy,
in which the stray field energy is expressed through \eqref{eq:Es2} or
\eqref{eq:Es3} appeared in some form in the engineering literature in
the context of finite element discretization of the magnetostatic
problems for ferromagnets (see, e.g.,
\cite{silvester70,demerdash80,coulomb81}) and is an extension of the
well-known variational principles for Maxwell's equations
\cite{monk,kuczmann}. Specifically, those minimization principles rely
on local constitutive relationships between the magnetic induction and
the magnetic field, which in the context of micromagnetics may be
obtained by first minimizing the micromagnetic energy written in terms
of the pair $(\mathbf M, \mathbf A)$ with respect to $\mathbf M$,
provided the exchange energy is neglected \cite{james90}. However, in
the full micromagnetics formulation the exchange energy plays a
crucial role, and, therefore, the variational formulation must include
a joint minimization of $E$ in $(\mathbf M, \mathbf A)$. Note that
while in the case of \eqref{eq:Es2} the minimization in $\mathbf A$
requires an additional constraint in the form of the Coulomb gauge,
the minimization in \eqref{eq:Es3} is unconstrained and automatically
enforces the Coulomb gauge for the minimizers. In fact, if one were to
minimize the expression in \eqref{eq:Es3} within the class in
\eqref{eq:Es2}, one would simply recover the problem in
\eqref{eq:Es2}, since for $\divv \mathbf A = 0$ the two energies
coincide, as can be easily seen via an integration by parts
\cite{du99}. On the other hand, the absence of the divergence-free
constraint, first noted in \cite{coulomb81}, makes the formulation in
\eqref{eq:Es3} clearly more attractive than that in \eqref{eq:Es2} and
opens up a way for an efficient numerical treatment of minimizers of
the micromagnetic energy.  In this paper, we put the above variational
principles on rigorous footing under natural assumptions.

Finally, we illustrate the usefulness of our results for analytical
studies of micromagnetics by applying the obtained variational
principles to the problem of finding the $\Gamma$-limit of the
micromagnetic energy in curved thin ferromagnetic shells. These
problems are interesting due to intrinsic symmetry-breaking mechanisms
coming from the non-zero curvature of the shell generating surfaces
(see \cite{di2019sharp,melcher2019curvature}; see also the recent
review \cite{streubel16}). Some results on this problem have been
previously obtained under technical assumptions on the geometry of the
domain occupied by the ferromagnet, see
\cite{carbou01,difratta16}. Here we show that using our approach these
restrictions can be easily removed, resulting in a leading-order
two-dimensional local energy functional in the spirit of Gioia and
James \cite{gioia97} formulated on two-dimensional surfaces, in which
the stray field energy reduces to the effective shape anisotropy term.

The paper is organized as follows. In section 2 we provide the
mathematical setup of the problem defining appropriate functional
spaces and proving some auxiliary results. In section 3 we prove
Theorem~\ref{mainthmscvec}, providing various characterizations of
the stray field energy. Section 4 is devoted to the proof of
Theorem~\ref{th:thinfilm}, characterizing the $\Gamma$-limit of the
micromagnetic energy of thin shells.

\section{Mathematical setup}
\label{s:setup}

In this section, we introduce the definitions and some useful facts
about the basic function spaces that will be needed in our
analysis. We would like to point out that the vectorial nature of the
problem associated with the demagnetizing field presents some
technical issues in the treatment of stationary Maxwell's equations
under minimal regularity assumptions on the magnetization. Although
some of the problems we are interested in can be investigated in a
potential-theoretic framework (see, e.g., \cite{dautray4,praetorius04,
  friedman1980mathematical, friedman1981mathematical}), here we rely
on their distributional formulations. Another technical issue has to
do with the fact that the problem is considered in the whole
space. For the sake of full generality, we consider the most general
distributional solutions of \eqref{B} and \eqref{eq:Maxwell} and show
that the resulting solutions do indeed belong to the natural energy
spaces, which is not obvious {\em a priori}.

We denote by $\Dp$ the space of distributions on $\RR^3$. Following
{\cite[p.~230]{dautray3}} and
{\cite[pp.~117--118]{dautray4}}, we define the homogeneous
Sobolev space
\begin{equation}
  \Bonegrad \assign \{ u \in \Dp \st \grad u \in \vLtwo \} .
  \label{eq:HSS}
\end{equation}
It is straightforward to show that the quotient space
\begin{equation}
 \Wonegradh \assign \Bonegrad / \RR
\end{equation}
is a Hilbert space for the $L^2$ gradient norm
$u \in\Wonegradh \mapsto \left\| \grad u \right\|_{L^2(\RR^3)}$, and that
$\Wonegradh$ is isometrically isomorphic to the {{weighted}} Sobolev
space $\left\{ u \in \Ltwow \st \grad u \in \vLtwo \right\}$, with
\begin{equation}
  \Ltwow \assign \left\{ u \in L^1_{\tmop{loc}}( \RR ) \st \omega
  u \in \Ltwo \right\}, \quad \omega (x) \assign \frac{1}{\sqrt{1 + | x |^2}}
  . \label{eq:L2w}
\end{equation}
In particular, up to an additive constant, every element of
$\Bonegrad$ is in $\Ltwow \subset L^1_{\tmop{loc}}(\R^3)$. For further
reference, we also define
$L^2_{\omega^{-1}}(\R^3) \assign \left\{ u \in L^1_{\tmop{loc}}( \RR )
  \st \omega^{-1} u \in \Ltwo \right\}$. The symbols
$L^2_{\omega}(\R^3; \R^3)$ and $L^2_{\omega^{-1}}(\R^3; \R^3)$ denote
the vector-valued analogs of the above spaces.


We denote by $\vDp$ the space of vector-valued distributions on
$\RR^3$. Also we denote by $\mathring{W}^1( \RR^3, \RR^3 )$ and
$\mathring{H}^1( \RR^3, \RR^3 ) \assign \mathring{W}^1( \RR^3, \RR^3 )
/ \RR^3$, the vector-valued counterparts of $\Bonegrad$ and
$\mathring{H}^1( \RR^3 )$, respectively, for which the same
considerations hold. Observe that
\begin{align}
  \label{adivcurl2}
  \| \nabla \tmmathbf a \|_{L^2(\RR^3)}^2 = \| \divv \tmmathbf a
  \|_{L^2(\RR^3)}^2 + \| \curl \tmmathbf a \|_{L^2(\RR^3)}^2
  \qquad \forall \tmmathbf a \in \mathring{H}^1( \RR^3; \RR^3 ),
\end{align}
which may be seen from the fact that for every $\tmmathbf a \in \vD$
we have
\begin{align}
  \| \nabla \tmmathbf a \|_{L^2(\RR^3)}^2=-\int_{\RR^3} 
  {\tmmathbf a} \cdot \lapl
  {\tmmathbf a}= \int_{\RR^3} {\tmmathbf a} \cdot \curl \curl
  {\tmmathbf a}-\int_{\RR^3} {\tmmathbf a} \cdot \grad \divv
  {\tmmathbf a}, 
\end{align}
and then arguing by density.

   In the spirit of {\eqref{eq:HSS}}, we also define the
  homogeneous Sobolev space
\begin{equation}
  \vBonecurl \assign \left\{ \tmmathbf{b} \in \vDp \st \curl \tmmathbf{b} \in
  \vLtwo \right\} . \label{HSScurl}
\end{equation}
Note that, $\vBonecurl$ is a subspace of $\vDp$, and that the functional
\begin{equation}
  \left| \, \cdot \, \right|_{\curl} : \tmmathbf{b} \in \vBonecurl \mapsto
  \int_{\RR^3} \left| \curl \tmmathbf{b} \right|^2  \label{eq:snBLcurl}
\end{equation}
is a seminorm on $\vBonecurl$. The kernel of $\left| \, \cdot \,
\right|_{\curl}$ consists of all curl-free distributions. Therefore, by
Poincar{\'e}-de~Rham lemma {\cite[{\tmabbr{p. }}355]{schwartz}},
\begin{eqnarray}
  \ker \left| \, \cdot \, \right|_{\curl} & \eqs & \nabla \Dp \; \equiv \;
  \left\{ \tmmathbf{b} \in \vDp \st \tmmathbf{b}= \grad v \text{ for some } v
  \in \Dp \right\} . 
\end{eqnarray}
We identify distributions which differ by a gradient field. The resulting
quotient space
\begin{equation}
  \vWonecurl \assign \vBonecurl / \nabla \Dp
\end{equation}
is a Hilbert space. Indeed, the following result holds.

\begin{proposition}
  The pair $( \vWonecurl, \left| \, \cdot \, \right|_{\curl}
  )$ forms a complete inner product space.
\end{proposition}

\begin{proof}
  Let $(\tmmathbf{b}_n)_{n \in \NN} \in \vWonecurl$ be a Cauchy
  sequence in $\vWonecurl$. This means that $( \curl \: \tmmathbf{b}_n
  )_{n \in \NN}$ is a Cauchy sequence in $\vLtwo$. Therefore, there
  exists $\tmmathbf{j} \in \vLtwo$ such that $\curl \: \tmmathbf{b}_n
  \rightarrow \tmmathbf{j}$ in $\vLtwo$. To prove completeness, it remains to
  show that $\tmmathbf{j}$ is in $\curl( \vDp )$. This is a
  consequence of Poincar{\'e}-de~Rham lemma {\cite[{\tmabbr{p.
  }}355]{schwartz}}. Indeed, as $\tmmathbf{j} \in \vLtwo$ we have,
  for every $\varphi \in \D$,
  \[ \dual{\divv \tmmathbf{j}}{\varphi} \eqs \int_{\RR^3} \tmmathbf{j} \cdot
     \grad \varphi \eqs \lim_{n \rightarrow \infty} \int_{\RR^3} \curl \:
     \tmmathbf{b}_n \cdot \grad \varphi \eqs 0, \]
  and therefore $\divv \tmmathbf{j}=\tmmathbf{0}$. Hence, $\curl
  \tmmathbf{b}=\tmmathbf{j}$ for some $\tmmathbf{b} \in \vDp$.
\end{proof}

We shall need the closed subspace of $\vWonecurl$ generated by
the limits of all divergence-free (solenoidal) and compactly
supported vector fields. To this end, we set
\begin{equation}
  \vDsol \assign \left\{ \tmmathbf{a} \in \vD \st \divv \tmmathbf{a} \equiv 0
  \right\} .
\end{equation}
\begin{remark} \label{rmk:Dsolinfinitedim}
  Since the set of harmonic functions in $\vD$ reduces to the null function,
  it is natural to concern about the cardinality of $\vDsol$. In that regard,
  we observe that the vector space $\vDsol$ is infinite-dimensional. Indeed,
  let $\rho : \RR \to \RR^+$ be in $\mathcal{D}( \RR )$ and
  suppose $\rho \equiv 1$ in a neighborhood of $0$. Also, let
  $\tmmathbf{\xi} \in C^{\infty}( \RR^3; \RR^3 )$ and consider
  the vector field
  \begin{equation}
  \tmmathbf{a} (x) \assign \rho (| x |)
    (\tmmathbf{\xi} (x) \times x) \qquad    x \in \RR^3.
  \end{equation}
  Clearly, $\tmmathbf{a} \in \mathcal{D}( \RR^3; \RR^3 )$ and,
  moreover, $\divv \tmmathbf{a} (x) \eqs \rho (| x |) \curl \tmmathbf{\xi}
  (x) \cdot x +( \grad [\rho (| x |)] \times \tmmathbf{\xi} (x)
  ) \cdot x$. Since $\grad [\rho (| x |)] = 0$ in a sufficiently small
  neighborhood of the origin, and outside that neighborhood one has $\grad
  [\rho (| x |)] = \rho' (| x |) x / | x |$, we get that $( \grad [\rho
  (| x |)] \times \tmmathbf{\xi} (x) ) \cdot x = 0$ everywhere in
  $\RR^3$. It then follows that $\divv \tmmathbf{a} (x) \eqs \rho (| x |) \curl
  \tmmathbf{\xi} (x) \cdot x$. As a consequence, for any curl-free vector
  field $\tmmathbf{\xi} \in C^{\infty}( \RR^3; \RR^3 )$, and
  any bump function $\rho$ we get $\divv \tmmathbf{a} \equiv 0$. This proves
  that $\vDsol$ is infinite-dimensional due to the arbitrary choices of
  $\rho$ and $\tmmathbf{\xi}$.
\end{remark}

We denote by {$\vWonecurlh$} the closure of $\vDsol$ in
$\vWonecurl$. We observe that, with
$\omega (x) \assign (1 + | x |^2)^{- 1 / 2}$, the following inequality
holds:
\begin{equation}
  \int_{\RR^3} | \tmmathbf{a} (x) |^2 \omega^2 (x) \mathd x \leqslant \; 4
  \int_{\RR^3} \left|  \curl \tmmathbf{a} (x) \right|^2 \mathd x \quad \forall
  \tmmathbf{a} \in \vDsol . \label{eq:PIforcurl}
\end{equation}
Indeed, \eqref{adivcurl2} and Hardy's inequality
{\cite[p.~296]{evanspde}} imply
$\| \omega \tmmathbf{a} \|_{L^2(\RR^3)}^2 \leqslant \; 4 \left\| \grad
  \tmmathbf{a} \right\|_{L^2(\RR^3)}^2$.

Our first observation is a regularity result on the structure of
{$\vWonecurlh$}. In what follows, we use the notation $[\tmmathbf{a}] \in
\vWonecurlh$ to denote the equivalence class which has $\tmmathbf{a} \in
\vBonecurl$ as representative; in other words, $[\tmmathbf{a}] \assign \left\{
\tmmathbf{a}+ \grad v \right\}_{v \in \Dp}$.

\begin{theorem}
  \label{thm:regcurl} The following
  statements hold:
  \begin{enumerate}
  \item[$\mathrm{(}i \mathrm{)}$] Let
    $[\tmmathbf{a}] \in \vWonecurlh$. There exists a unique
    representative
    $\tmmathbf{a}^{\star} \in [\tmmathbf{a}] \cap \mathring{H}^1(
    \RR^3; \RR^3 )$ which is divergence-free. In particular,
    $\tmmathbf{a}^{\star}$ is the unique divergence-free
    representative of $[\tmmathbf a]$ that belongs to $\vLtwow$.
    \smallskip
  \item[$\mathrm{(}ii \mathrm{)}$] If $[\tmmathbf{a}]\in \vWonecurlh$
    has a representative $\tmmathbf{\jmath} \in \vLtwo$, then also
    $\tmmathbf{a}^{\star}$ belongs to $\vLtwo$. Precisely,
    $\tmmathbf{a}^{\star}$ can be decomposed in the form
    \begin{equation}
      \tmmathbf{a}^{\star} =\tmmathbf{\jmath}+ \grad v_{\tmmathbf{\jmath}},
    \end{equation}
    with $v_{\tmmathbf{\jmath}}$ the unique solution, in $\mathring{H}^1(
    \RR^3 )$, of the Poisson equation $- \lapl v_{\tmmathbf{\jmath}} =
    \divv \tmmathbf{\jmath}$.
\smallskip    
\item[$\mathrm{(}iii \mathrm{)}$] If
  $\tmmathbf{a}_\circ\in \mathring{H}^1( \RR^3; \RR^3 )$ and
  $\divv \tmmathbf{a}_\circ=0$ then
  $[\tmmathbf{a}_\circ]\in \vWonecurlh$ and
  $\tmmathbf{a}_\circ=\tmmathbf{a}^{\star}$.
  \end{enumerate}
\end{theorem}

\begin{proof}
  $\mathrm{(}i \mathrm{)}$ Let $[\tmmathbf{a}] \in \vWonecurlh$, and
  $\tmmathbf{a}_n \in \vDsol$ be such that
  $\tmmathbf{a}_n \rightarrow \tmmathbf{a}$ in $\vWonecurl$. Clearly,
  $[\tmmathbf{a}] \in \vWonecurlh$ and
  $( \curl \tmmathbf{a}_n )_{n \in \NN}$ is Cauchy in
  $\vWonecurlh$. Since $\vLtwow$ is a complete space, by
  {\eqref{eq:PIforcurl}}, there exists
  $\tmmathbf{a}^{\star} \in \vLtwow$ such that
  $\tmmathbf{a}_n \rightarrow \tmmathbf{a}^{\star}$ in $\vLtwow$.
  Therefore,
  \begin{alignat}{2}
    0 \eqs \divv \tmmathbf{a}_n & \rightarrow \divv
    \tmmathbf{a}^{\star}
    \eqs 0  & &\quad\text{ in } \vDp, \\
    \curl \tmmathbf{a}_n & \rightarrow \curl \tmmathbf{a}^{\star} &&
    \quad\text{
      in } \vDp, \\
    \curl \tmmathbf{a}_n \; &\rightarrow \curl \tmmathbf{a} &&
    \quad\text{ in } \vDp .
  \end{alignat}
  This means that $\curl  (\tmmathbf{a}^{\star} -\tmmathbf{a}) = 0$ and,
  therefore, that in any equivalence class $[\tmmathbf{a}] \in \vWonecurlh$
  there exists a divergence-free vector field $\tmmathbf{a}^{\star} \in
  \vLtwow$. Note that $\tmmathbf{a}^{\star} \in \vLtwow$ is then necessarily
  unique. Indeed, if $\tmmathbf{\jmath}^{\star} \in \vLtwow$ is another
  divergence-free representative, then $\curl \tmmathbf{a}^{\star} = \curl
  \tmmathbf{\jmath}^{\star}$ and $\divv \tmmathbf{a}^{\star} = \divv
  \tmmathbf{\jmath}^{\star} = 0$. This implies that 
  \begin{align}
    \mathbf 0=\nabla(\divv(\tmmathbf{a}^{\star} -
    \tmmathbf{j}^{\star})) -
    \curl(\curl(\tmmathbf{a}^{\star}-\tmmathbf{j}^{\star}))= \lapl 
    (\tmmathbf{a}^{\star} -\tmmathbf{\jmath}^{\star}) \hbox{ in } \vDp,
  \end{align}
  and in view of
  $\tmmathbf{a}^{\star}- \tmmathbf{j}^{\star} \in \vLtwow$ we have
  $\lapl (\tmmathbf{a}^{\star} -\tmmathbf{\jmath}^{\star}) = \mathbf
  0$ in the sense of tempered distributions
  $\tmmathbf{\mathcal{S}}'( \RR^3 )$. Therefore, by Liouville's
  theorem {\cite[p.~41]{eskin}}, it follows that
  $\tmmathbf{a}^{\star} -\tmmathbf{\jmath}^{\star}$ is a polynomial
  vector field. We conclude by observing that the only polynomial
  vector field in $\vLtwow$ is the zero vector field.
  
  It remains to prove that
  $\tmmathbf{a}^{\star} \in \mathring{H}^1( \RR^3; \RR^3 )$. We
  observe that since $\divv \tmmathbf{a}^{\star} = 0$, if we set
  $\tmmathbf{b}^{\star} \assign \curl \tmmathbf{a}^{\star}$, then
  $\tmmathbf{a}^{\star}$ is a solution of the vector Poisson equation
  $- \lapl \tmmathbf{a}= \curl \tmmathbf{b}^{\star}$. Also, since
  $\tmmathbf{b}^{\star} \in \vLtwo$, we have that
  $\curl \tmmathbf{b}^{\star}$ generates a linear and continuous
  functional on $\mathring{H}^1( \RR^3; \RR^3 )$, and therefore, by
  Riesz representation theorem, there exists a unique
  $\tmmathbf{a} \in \mathring{H}^1( \RR^3; \RR^3 )$ such that
  $- \lapl \tmmathbf{a}= \curl \tmmathbf{b}^{\star}$. But this implies
  that $\tmmathbf{a}-\tmmathbf{a}^{\star}$ is a harmonic $\vLtwow$
  vector field; therefore, necessarily
  $\tmmathbf{a}^{\star} =\tmmathbf{a} \in \mathring{H}^1 ( \RR^3;
  \RR^3 )$.  \medskip
 
 \noindent $\mathrm{(}ii \mathrm{)}$ If
 $\tmmathbf{\jmath} \in [\tmmathbf{a}] \cap \vLtwo$ then there exists
 $v_{\tmmathbf{\jmath}} \in \nabla \Dp$ such that
 $\tmmathbf{\jmath}-\tmmathbf{a}^{\star} = - \grad
 v_{\tmmathbf{\jmath}}$.  Hence,
  \begin{equation}
    - \lapl v_{\tmmathbf{\jmath}} \eqs \divv
    (\tmmathbf{\jmath}-\tmmathbf{a}^{\star}) \eqs \divv \tmmathbf{\jmath},
  \end{equation}
  and the previous equation admits a unique solution $v_{\tmmathbf{\jmath}}
  \in \mathring{H}^1( \RR^3 )$ by Riesz representation theorem for the
  dual of a Hilbert space.
  
  \medskip
  \noindent $\mathrm{(}iii \mathrm{)}$ Let
  $\tmmathbf{a}_\circ\in \mathring{H}^1( \RR^3; \RR^3 )$ be such that
  $\divv \tmmathbf{a}_\circ=0$. The variational equation
  \begin{equation}\label{eq:testinH1sol0}
    \int_{\RR^3} \curl  \tmmathbf{a} \cdot \curl
    \tmmathbf{\varphi}^\star = 
    \int_{\RR^3}  \curl \tmmathbf{a}_\circ \cdot \curl
    \tmmathbf{\varphi}^\star  \quad 
    \forall \tmmathbf{\varphi}^\star \in \vWonecurlh . 
  \end{equation}
  has a unique solution $[\tmmathbf{a}]\in \vWonecurlh$ because
  $\curl \curl \tmmathbf{a}_\circ$ can be identified with an element
  of $\mathring{H}^{-1}_{\textrm{sol}}(\curl,\RR^3) $. In particular,
  testing against functions of the type
  $\tmmathbf{\varphi}^\star:=\curl \tmmathbf{\varphi}$ with
  $ \tmmathbf{\varphi}\in \vD$, we get that
  \begin{equation}
    \curl (\curl \curl (\tmmathbf{a}-\tmmathbf{a}_\circ)) =
    \tmmathbf{0}\quad \text{in }\vDp.
  \end{equation}
  At the same time, by the result in point $(i)$ we have that
    $\tmmathbf{a}^{\star}\in \vLtwow$ is the unique divergence-free
  representative belonging to
  $[\tmmathbf{a}] \cap \mathring{H}^1( \RR^3;\RR^3 )$. This implies
  that
  \begin{equation}
    -\lapl( \curl (\tmmathbf{a}^{\star}-\tmmathbf{a}_\circ)) =
    \tmmathbf{0}\quad \text{in }\vDp, 
  \end{equation}
  with $ \curl (\tmmathbf{a}^{\star}-\tmmathbf{a}_\circ)\in
  \vLtwo$. Therefore
  $\curl (\tmmathbf{a}^{\star}-\tmmathbf{a}_\circ)=\tmmathbf{0}$,
  which means $\tmmathbf{a}_\circ\in [\tmmathbf{a}^{\star}]$. Again,
  by the uniqueness of the divergence-free representative we conclude
  that $\tmmathbf{a}_\circ=\tmmathbf{a}^{\star}$.
\end{proof}

\section{Magnetostatics}
\label{s:var}

We begin by non-dimensionalizing the micromagnetic energy, using the
exchange length
$\ell_\mathrm{ex} := \sqrt{2 A / (\mu_0 M_\mathrm{s}^2)}$ as the unit
of length. Introducing the normalized magnetization vector
$\m(\mathbf r) := \mathbf M(\ell_\mathrm{ex} \mathbf r) /
M_\mathrm{s}$ depending on the dimensionless position vector
${\mathbf r}$, the quality factor $Q := 2 K / (\mu_0 M_\mathrm{s}^2)$
associated with crystalline anisotropy, and
\begin{align}
  \label{eq:scale}
  \hd_\mathrm{d} = {\mathbf H_\mathrm{d} \over M_\mathrm{s}}, \qquad
  \ha = {\mathbf 
  H_\mathrm{a} \over M_\mathrm{s}}, \qquad \mathcal E(\m) = {E(\mathbf
  M) \over 2 A \ell_\mathrm{ex}}, 
\end{align}
we can write the micromagnetic energy in dimensionless form as
\begin{align}
  \label{eq:EE}
  \mathcal E(\m) := \frac12 \int_\Omega |\nabla \m|^2 + {Q \over 2}
  \int_\Omega \Phi(\m) - \int_\Omega \ha \cdot \m - \frac12
  \int_\Omega \hd_\mathrm{d} \cdot \m, 
\end{align}
where $\Omega$ was appropriately rescaled and the symbol $d^3 r$ is
omitted from all the integrals from now on for simplicity of
presentation. The rescaled demagnetizing field $\hd_\mathrm{d}$ and
the associated rescaled magnetic induction $\bd_\mathrm{d}$ solve
\begin{alignat}{2}
  &\curl \hd \eqs \tmmathbf{0} &\quad \text{in }
  \RR^3,   \label{eq:MA1}\\
  &\tmop{div} \bd \eqs 0 &\quad \text{in } \RR^3,  \label{eq:MA2}\\
  &\bd \eqs \hd + \m &\quad \text{in } \RR^3 .
  \label{eq:MA3}
\end{alignat}
In turn, the corresponding rescaled scalar potential $\ud_\mathrm{d}$ and
vector potential $\ad_\mathrm{d}$ are related to their unscaled counterparts
via
\begin{align}
  \label{eq:udad}
  \ud_\mathrm{d}(\mathbf r) := {U_\mathrm{d}(\ell_\mathrm{ex} \mathbf
  r) \over M_\mathrm{s} \ell_\mathrm{ex}}, 
  \qquad \ad_\mathrm{d}(\mathbf r) := {\mathbf A_\mathrm{d}
  (\ell_\mathrm{ex} \mathbf r) \over 
  \mu_0 M_\mathrm{s} \ell_\mathrm{ex}},
\end{align}
so that $\bd_\mathrm{d} = \curl \ad_\mathrm{d}$ and
$\hd_\mathrm{d} = -\nabla \ud_\mathrm{d}$. Finally, the rescaled stray
field energy is
\begin{align}
  \label{eq:Ess}
  \mathcal E_\mathrm{s}(\m) := -\frac12 \int_{\R^3} \hd_\mathrm{d}
  \cdot \m. 
\end{align}
where $\hd_\mathrm{d}$ is understood as a function of $\m$ uniquely
determined by the solution of
{\tmem{{\eqref{eq:MA1}}-{\eqref{eq:MA3}}}} (for a precise statement,
see below).

Throughout the rest of this paper, we suppress the subscript ``$d$''
everywhere to avoid cumbersome notations. However, whenever needed we
utilize the subscript $\m$ to explicitly indicate the dependence of
the associated quantities on a given magnetization $\m$, so there
should be no confusion. The main result of this section is
Theorem~\ref{mainthmscvec}. We remark that all the assumptions
  of this theorem are satisfied in the context of micromagnetics when
  the ferromagnet occupies a bounded domain.

\begin{theorem}
  \label{mainthmscvec}Let $\m \in \vLtwo$. The following assertions hold: 
  \begin{enumerate}[(i)]
    \item  There exists a unique magnetic scalar
    potential $u_{\m} \in \mathring{H}^1( \RR^3 )$ such that
    \begin{equation}
      \label{eq:hmum}
      \hd_{\m} \assign - \grad u_{\m}, \quad \bd_{\m} \assign 
      \hd_{\m} + \m,
    \end{equation}
    is a solution of {\tmem{{\eqref{eq:MA1}}-{\eqref{eq:MA3}}}} in
    $\vLtwo \times \vLtwo$. The stray field energy is given through
    the following maximization problem:
    \begin{equation}
      \mathcal E_\mathrm{s}(\m) = \max_{u \in \mathring{H}^1( \RR^3 )}
      \mathcal{W}(\m, u ), \quad \quad \mathcal{W}(
      \m, u ) \assign \int_{\RR^3} \grad u \cdot \m - \frac{1}{2}
      \int_{\RR^3} \left| \grad u \right|^2, \label{eq:minscpotthm}
    \end{equation}
    whose unique solution coincides with $u_{\m}$.  Moreover, if
   $\m\in \vLtwoww$ then $u_{\m} \in {H}^1( \RR^3 )$.
    \smallskip
    
   \item  There exists a unique magnetic vector
    potential $\left[ \tmmathbf{a}_{\m} \right] \in \vWonecurl$ such that
    \begin{equation}
      \bd_{\m}' \assign \curl [\tmmathbf{a}_{\m}], \quad \hd_{\m}' \assign
      \bd_{\m}' - \m,
    \end{equation}
    is a solution of {\tmem{{\eqref{eq:MA1}}-{\eqref{eq:MA3}}}} in
    $\vLtwo \times \vLtwo$.  The stray field energy is given through
    the following minimization problem:
    \begin{equation}
      \mathcal E_\mathrm{s}(\m) = \underset{[\tmmathbf{a}] \in
        \vWonecurl}{\min} \mathcal{V}_{\curl} 
      ( \m, [\tmmathbf{a}] ), \quad \quad \mathcal{V}_{\curl}( \m,
      [\tmmathbf{a}] ) \assign \frac{1}{2} \int_{\RR^3} \left| \curl 
        [\tmmathbf{a}] - \m \right|^2 , \label{eq:minvecpotthm}
    \end{equation}
    whose unique solution coincides with
    $\left[ \tmmathbf{a}_{\m} \right]$.

    Moreover, if $\m\in \vLtwoww$ then there exists a unique
    representative
    $\tmmathbf a^\star_{\m}\in\left[ \tmmathbf{a}_{\m} \right]$
    satisfying the Coulomb gauge conditions
    \begin{equation}
      \tmmathbf{a}_{\m}^{\star} \in \vLtwo, \quad \divv
      \tmmathbf{a}_{\m}^{\star} = 0. \label{eq:CGauge}
    \end{equation}
    The representative $\tmmathbf{a}^\star_{\m}$ belongs to
    $ H^1(\RR^3;\RR^3)$ and can be characterized as the unique
    solution in $ \mathring{H}^1( \RR^3; \RR^3 )$ of the vector
    Poisson equation
    \begin{equation} -\lapl \tmmathbf a^\star_{\m}=\curl \m\quad
      \text{in }\mathring{H}^{-1}( \RR^3; \RR^3 ).
\end{equation}
Equivalently, $\tmmathbf a^\star_{\m}$ can be characterized as the
unique solution in $\vWonecurlh $ of the variational equation
\begin{equation}\label{eq:testinH1sol}
  \int_{\RR^3} \curl  \tmmathbf{a}^\star_{\m} \cdot \curl
  \tmmathbf{\varphi}^\star = 
  \int_{\RR^3}  \m \cdot \curl \tmmathbf{\varphi}^\star  \quad
  \forall \tmmathbf{\varphi}^\star \in \vWonecurlh . 
\end{equation}
    
    \smallskip
  \item We have
    \begin{align}
      \label{eq:mmprime}
      \hd_{\m} = \hd_{\m}', \qquad \bd_{\m} = \bd_{\m}', \qquad 
      \mathcal E_\mathrm{s}(\m) = \frac12 \int_{\R^3} |\hd_{\m}|^2.       
    \end{align}

  \item If $\m\in \vLtwoww$, the stray field energy admits the
    following representation: 
    \begin{align}
      \mathcal E_\mathrm{s}(\m) = \min_{\tmmathbf{a} \in
      \mathring{H}^1( \RR^3; \RR^3 )} 
      \mathcal{V}( \m, \tmmathbf{a} ),  \quad  \mathcal{V}( \m,
      \tmmathbf{a} ) \assign \frac{1}{2} 
      \int_{\RR^3} | \nabla \tmmathbf{a} |^2 + \frac12 \int_{\RR^3}
      \left| \m \right|^2 - \int_{\RR^3} \m \cdot \curl  
      \ad,
      \label{eq:minvecpotnodiv}
    \end{align}
    and the unique minimizer of $\mathcal V(\m, \cdot)$ coincides with
    $\tmmathbf{a}_{\m}^\star$.
  \end{enumerate}
\end{theorem}

\begin{proof}
  (\tmem{i}) We start with an observation that holds under minimal
  regularity assumptions. Let $\m \in \vDp$. If a solution
  $( \hd_{\m}, \bd_{\m} ) \in \vDp \times \vDp$ of
  {\eqref{eq:MA1}}-{\eqref{eq:MA3}} exists, then
  $\curl \hd_{\m} \eqs \tmmathbf{0}$ distributionally. Therefore,
  according to Poincar{\'e}-de~Rham lemma
  \cite[{\tmabbr{p. }}355]{schwartz}, there exists a magnetostatic
  potential $u_{\m} \in \Dp$ such that $\hd_{\m} = - \grad
  u_{\m}$. But then, from {\eqref{eq:MA2}} and {\eqref{eq:MA3}}, we
  get that $u_{\m}$ is a particular solution of the Poisson equation
  \begin{equation}
    \lapl u_{\m} = \divv  \m \quad \text{in }  \Dp .
    \label{eq:PoissoninDp}
  \end{equation}
  Conversely, if $u_{\m}$ is a particular solution of
  {\eqref{eq:PoissoninDp}}, then the general solution of the magnetostatic
  equations is given by
  \begin{equation}
    \hd_{\m} \assign - \grad u_{\m} + \grad v_{0}, \quad \bd_{\m}
    \assign \hd_{\m} + \m, \label{eq:MAhsol}
  \end{equation}
  for an arbitrary harmonic distribution $v_{0} \in \Dp$. Indeed,
  defining $\hd_{\m} \assign - \grad u_{\tmmathbf{m}}$ and
  $\bd_{\m} \assign \hd_{\m} + \m $ we have that
  $(\hd_{\m}, \bd_{\m})$ is a solution of
  {\eqref{eq:MA1}}-{\eqref{eq:MA3}}, and any other demagnetizing field
  differs by a gradient distribution. Taking the divergence of the
  first equation in {\eqref{eq:MAhsol}} we get that $v_{0} \in \Dp$ is
  necessarily harmonic.
  
  Now, for $\m \in \vLtwo$ we have that $\divv \m$ generates a linear
  continuous functional on $\mathring{H}^1( \RR^3 )$ and, therefore,
  by Riesz representation theorem there exists a unique
  $u_{\m} \in \mathring{H}^1( \RR^3 )$ such that
  \begin{equation}
    \int_{\RR^3} \grad u_{\m} \cdot \grad \varphi \eqs \int_{\RR^3} \m
    \cdot \grad \varphi \qquad \forall \varphi \in \mathring{H}^1(
    \RR^3 ). \label{eq:weakum} 
  \end{equation}
  Hence, setting
  \begin{equation}
    \hd_{\m} \assign - \grad u_{\m}, \quad \bd_{\m} \assign 
    \hd_{\m} + \m 
  \end{equation}
  we get a solution $( \hd_{\m}, \bd_{\m} ) \in \vLtwo \times \vLtwo$
  of {\eqref{eq:MA1}}-{\eqref{eq:MA3}}. Also, note that $u_{\m}$ is
  the unique magnetostatic potential which gives a demagnetizing field
  in $\vLtwo$. Indeed, if $- \grad u_{\m} + \grad v_{0} \in \vLtwo$
  with $v_{0}$ harmonic, then, according to Liouville's theorem
    $\grad v_{0} =\tmmathbf{0}$. Finally, a standard argument gives
  that $u_{\m}$ coincides with the unique solution of the
  {\tmem{maximization problem}} {\eqref{eq:minscpotthm}}.
  
  Now, if $\m\in \vLtwoww$ then $\m$ generates a continuous
  linear functional on $\mathring{H}^1( \RR^3; \RR^3 )$. Indeed, by
  Hardy's inequality, for every
  $\tmmathbf{\varphi} \in \mathring{H}^1( \RR^3; \RR^3 )$ we have
\begin{equation}
  \int_{\RR^3} | \m \cdot \tmmathbf{\varphi} | \leqslant \|
  \omega^{-1}\m \|_{L^2(\RR^3)}\, \| \omega
  \tmmathbf{\varphi}\|_{{L^2(\RR^3)}} \leqslant 4 \| \omega^{-1}\m
  \|_{L^2(\RR^3)}\, 
  \|\nabla\tmmathbf{\varphi}\|_{{L^2(\RR^3)}}. 
\end{equation}  
Therefore, by Riesz representation theorem there exists a unique
$\tmmathbf\psi_{\m}\in \mathring{H}^1( \RR^3; \RR^3 )$ such that
$-\lapl \psi_{\m} = \m $. We set $u_{\m} := -\divv \psi_{\m}$. Note
that $u_{\m} \in \Ltwo$ and satisfies the equation
\begin{equation}
  \lapl u_{\m} = -\divv \grad (\divv \psi_{\m})=-\divv \lapl
  \psi_{\m}=\divv \m  \quad {\text{in }  \Dp} .
\end{equation}
This implies that
$u_{\m}\in L^2(\RR^3) \cap \mathring{H}^1( \RR^3)=H^1(\RR^3)$. 
 
 \medskip
 \noindent (\tmem{ii}) Once again, we start with an observation that
 is valid under minimal regularity assumptions. Let $\m \in \vDp$. If
 a solution $( \hd_{\m}, \bd_{\m} ) \in \vDp \times \vDp$ of
 {\eqref{eq:MA1}}-{\eqref{eq:MA3}} exists, then
 $\divv \bd_{\m} \eqs \tmmathbf{0}$ distributionally. Therefore, it
 follows from Poincar{\'e}-de~Rham lemma that there exists a vector
 potential $\tmmathbf{a}_{\m} \in \vDp$ such that
 $\bd_{\m} = \curl \tmmathbf{a}_{\m}$. But then, from {\eqref{eq:MA1}}
 and {\eqref{eq:MA3}}, we get that $\tmmathbf{a}_{\m}$ is a particular
 solution of the double-curl equation
  \begin{equation}
    \curl  \curl \tmmathbf{a}_{\m} \eqs  \curl  \m \quad \text{in
    }  \vDp . \label{eq:PoissoninDpvp}
  \end{equation}
  Conversely, assume that $\bar{\tmmathbf{a}}_{\m}$ is a particular
  solution of {\eqref{eq:PoissoninDpvp}}. We claim that the general
  solution of {\eqref{eq:MA1}}-{\eqref{eq:MA3}} is given by
  \begin{equation}
    \bd_{\m} \assign \curl  \bar{\tmmathbf{a}}_{\m} + \grad v_{0},
    \quad \hd_{\m} \assign \bd_{\m} - \m
  \end{equation}
  for an arbitrary harmonic distribution $v_{0} \in \Dp$. Indeed, the
  assignment
  $\bar{\tmmathbf{b}}_{\m} \assign \curl \bar{\tmmathbf{a}}_{\m}$ and
  $\bar{\hd}_{\m} \assign \bar{\tmmathbf{b}}_{\m} - \m$ gives a
  particular solution of {\eqref{eq:MA1}}-{\eqref{eq:MA3}}. Moreover,
  any other vector field $\tmmathbf{b}$ satisfying
  {\eqref{eq:MA1}}-{\eqref{eq:MA3}} must differ from
  $\bar{\tmmathbf{b}}_{\m}$ by a curl distribution, i.e., we have
  \begin{equation}
    \tmmathbf{b}_{\tmmathbf{m}} \assign \curl ( \tmmathbf{a}_{0} +
    \bar{\tmmathbf{a}}_{\m} ), \quad \hd_{\m} \assign 
    \bd_{\m} - \m \eqs \curl ( \tmmathbf{a}_{0} +
    \bar{\tmmathbf{a}}_{\m} ) - \m , \label{eq:MAbsol}
  \end{equation}
  for some $\tmmathbf{a}_{0} \in \vDp$.  Taking the {\curl} of the
  second equation in {\eqref{eq:MAbsol}}, we get
  \begin{equation}
    \curl  \curl ( \tmmathbf{a}_{0} +
    \bar{\tmmathbf{a}}_{\m} ) - \curl  \m =\tmmathbf{0},
  \end{equation}
  and from the definition of $\bar{\tmmathbf{a}}_{\m}$ we obtain that
  $\curl \curl \bar{\tmmathbf{a}}_{0} =\tmmathbf{0}$. It follows that
  $\curl \bar{\tmmathbf{a}}_{0} = \grad v_{0}$ for some
  $v_{0} \in \Dp$. In particular, $v_{0}$ is a harmonic distribution.
  
  Now, for $\m \in \vLtwo$ we have that $\curl \m$ generates a
  linear continuous functional on $\vWonecurl$ and, therefore, by
  Riesz representation theorem there exists a unique
  $\left[ \tmmathbf{a}_{\m} \right] \in \vWonecurl$ such that
  \begin{equation}
    \int_{\RR^3} \curl [\tmmathbf{a}_{\m}] \cdot \curl \tmmathbf{\psi} \eqs
    \int_{\RR^3}  \m \cdot \curl \tmmathbf{\psi} \quad \forall
    \tmmathbf{\psi} \in \vWonecurl \label{eq:weakformcurl} .
  \end{equation}
  Hence, setting
  \begin{equation}
    \bd_{\m}' \assign \curl [\tmmathbf{a}_{\m}], \quad \hd_{\m}' \assign
    \bd_{\m}' - \m, 
  \end{equation}
  we get a solution
  $( \hd_{\m}', \bd_{\m}' ) \in \vLtwo \times \vLtwo$ {of
    {\eqref{eq:MA1}}-{\eqref{eq:MA3}}}.  Note that $\tmmathbf{a}_{\m}$
  is the unique magnetostatic potential which gives
  {$\tmmathbf{b}_{\m} \in \vLtwo$}. Indeed, if
  $\curl \tmmathbf{a}_{\m} + \grad v_{0} \in \vLtwo$ and $v_{0}$ is
  harmonic, then necessarily $\grad v_{0} =\tmmathbf{0}$.  From the
  preceding considerations, it is clear that the variational
  characterization {\eqref{eq:minvecpotthm}} holds.

  Next, as in the proof of (\emph{i}), for $\m\in \vLtwoww$
  there exists a unique
  $\tmmathbf\psi_{\m}\in \mathring{H}^1( \RR^3; \RR^3 )$ such that
  $-\lapl \psi_{\m} = \m $. We set
  $\tmmathbf{a}^\star_{\m} := \curl \psi_{\m}$. Note that
  $\tmmathbf{a}^\star_{\m} \in \vLtwo$ and, by construction,
  $\divv \tmmathbf{a}^\star_{\m}=0$. Also, $\tmmathbf{a}^\star_{\m}$
  satisfies the equation
  \begin{equation} \label{eq:curlomegam} \curl \tmmathbf{a}^\star_{\m}
    = \curl \curl \tmmathbf\psi_{\m} =\m + \grad \divv
    \tmmathbf\psi_{\m} .
\end{equation}
But $\divv \tmmathbf\psi_{\m}\in L^2(\RR^3)$ satisfies
$-\lapl(\divv \tmmathbf\psi_{\m})=\divv \m$, and therefore
$\grad \divv \tmmathbf\psi_{\m}\in L^2(\RR^3;\RR^3)$. Overall, from
\eqref{eq:curlomegam}, we infer that $[\tmmathbf{a}^\star_{\m}]$ is an
element of $\vWonecurl$ satisfying \eqref{eq:weakformcurl}. It follows
that $[ \tmmathbf{a}^\star_{\m}]=[\tmmathbf{a}_{\m}]$ and
$\divv \tmmathbf{a}^\star_{\m}=0$. Also, from \eqref{eq:curlomegam} we
know that $\tmmathbf{a}^\star_{\m}$ solves the equation
$-\lapl \tmmathbf{a}^\star_{\m} = \curl \m$ with data $\curl \m$ in
$\mathring{H}^{-1}( \RR^3; \RR^3 )$. Hence,
$\tmmathbf{a}^\star_{\m}\in {H}^{1}( \RR^3; \RR^3 )$.

Finally, if $[\tmmathbf{a}^{\star\star}_{\m}] \in \vWonecurlh$ is the
unique solution of \eqref{eq:testinH1sol} and
$\tmmathbf{a}^{\star\star}_{\m} \in L^2_\omega(\RR^3;\RR^3)$ its
unique divergence-free representative, testing against
$\tmmathbf{\varphi}^\star=\curl \tmmathbf{\varphi}$ with
$\tmmathbf{\varphi} \in \mathcal{D}(\RR^3;\RR^3)$ we get
\begin{equation}
  \curl \curl \tmmathbf{a}^{\star\star}_{\m}=\curl \m +\grad v_0\quad
  \text{in }\mathcal{D}'(\RR^3;\RR^3), 
\end{equation}
for some harmonic polynomial $v_0$. Therefore, since
  $\tmmathbf{a}^{\star\star}_{\m}$ is divergence-free, we have
\begin{equation}
  -\lapl(\curl(\tmmathbf{a}^{\star\star}_{\m}-\tmmathbf{a}^{\star}_{\m}))
  =\tmmathbf{0},  
\end{equation}
with
$\curl(\tmmathbf{a}^{\star\star}_{\m}-\tmmathbf{a}^{\star}_{\m})\in
\vLtwo$. But this means that
$\tmmathbf{a}^{\star\star}_{\m}=\tmmathbf{a}^{\star}_{\m}+\grad v$
with $v$ harmonic and $\grad v\in L^2_\omega(\RR^3;\RR^3)$. Therefore
$\grad v=\tmmathbf{0}$. This concludes the proof of (\emph{ii}).  

  \noindent (\tmem{iii}) The first two equalities in
  \eqref{eq:mmprime} follow from the uniqueness of solutions of
  {\eqref{eq:MA1}}-{\eqref{eq:MA3}} in $L^2(\R^3; \R^3)$. The third
  equality in \eqref{eq:mmprime} follows from \eqref{eq:hmum} and
  \eqref{eq:minscpotthm}.

  \noindent (\tmem{iv}) From \eqref{eq:testinH1sol} it is clear that
  \begin{equation}
    \mathcal E_\mathrm{s}(\m) = \min_{[\tmmathbf{a}^{\star}] \in
      \vWonecurlh} \mathcal{V}_{\curl}( 
    \m, [\tmmathbf{a}^{\star}] ),
  \end{equation}
  where we noted that the minimum above is attained because
  $\vWonecurlh$ is a closed subspace of the Hilbert space
  $\vWonecurl$.  Since $\mathring{H}^1 ( \RR^3; \RR^3 )$ can be
  identified with a subset of $\vWonecurl$, and {\eqref{eq:CGauge}}
  holds, it is sufficient to show that
  \begin{equation}
    \min_{\tmmathbf{a} \in \mathring{H}^1( \RR^3; \RR^3 )} \mathcal{V}
    ( \m, \tmmathbf{a} ) \; \leqslant \mathcal E_\mathrm{s}(\m). 
    \label{eq:lowboundenV}
  \end{equation}
  To this end, we observe that if
  $\left[ \tmmathbf{a}^{\star}_{\m} \right] \in \vWonecurlh$ minimizes
  $\mathcal{V}_{\curl}( \m, [\tmmathbf{a}^{\star}] )$, then, without
  loss of generality, we can assume that $\tmmathbf{a}^{\star}_{\m}$
  is the unique representative satisfying the Coulomb gauge regularity
  conditions {\eqref{eq:CGauge}}. But then, since
  $\divv \tmmathbf{a}^{\star}_{\m} = 0$, by \eqref{adivcurl2} we
  have
  \begin{equation}
    \tmmathbf{a}^{\star}_{\m} \in \mathring{H}^1( \RR^3; \RR^3 ), \quad
    \mathcal{V}( \m, \tmmathbf{a}^{\star}_{\m} )
    =\mathcal{V}_{\curl}( \m, \left[
      \tmmathbf{a}^{\star}_{\m} \right] ),
  \end{equation}
  and this implies {\eqref{eq:lowboundenV}}.
\end{proof}


\begin{remark}
  The weight $\omega$ in the assumptions on $\m$ imposes the behavior
  at infinity of the magnetostatic potential $u_{\m}$. Note that in
  general $u_{\m}$ does not belong to $H^1(\R^3)$ if
  $\m \in L^2(\R^3;\R^3)$. To see this consider $\m = -\nabla u$ with
  $u \in \mathring{H}^1(\R^3) \setminus H^1(\R^3)$. However, it is
  known that $u \in H^1(\R^3)$ provided $\m \in L^2(\R^3; \R^3)$ has
  compact support \cite{james90, praetorius04}. The above theorem
  gives a generalization of this result to a wider class of functions
  $\m \in L^2_{\omega^{-1}} (\R^3; \R^3)$.
\end{remark}

\begin{remark}
  If $u_{\tmmathbf{m}'}$ is the unique weak solution of
  $\lapl u_{\tmmathbf{m}'} = \divv \tmmathbf{m}'$, with
  $\tmmathbf{m}' \in \vLtwo$, then testing against
  $\varphi \assign u_{\tmmathbf{m}'}$ in the weak formulation of
  $\lapl u_{\m} = \divv \m$, and testing against
  $\varphi \assign u_{\m}$ in the weak formulation of
  $\lapl u_{\tmmathbf{m}'} = \divv \tmmathbf{m}'$, we get the
  so-called {\tmem{reciprocity relations}}
  \begin{equation}\label{eq:eequality1}
    \int_{\RR^3} \hd_{\m} \cdot \hd_{\tmmathbf{m}'} \eqs - \int_{\RR^3} \m
    \cdot \hd_{\tmmathbf{m}'} \eqs - \int_{\RR^3} \hd_{\m} \cdot
    \tmmathbf{m}'. 
  \end{equation}
  Thus, the operator
  $\mathcal H : \m \in \vLtwo \mapsto \hd_{\m} \in \vLtwo$ is
  self-adjoint, and for $\m =\tmmathbf{m}'$ we recover the expression
  of $\mathcal{E}_\mathrm{s}(\m)$ in \eqref{eq:mmprime}. Furthermore,
  $\mathcal H$ has unit norm, as can be seen from
  \begin{align}
    \left\| \hd_{\m} \right\|_{L^2(\RR^3)} \leqslant \| \m
    \|_{L^2(\RR^3)} \qquad \forall \m \in L^2(\R^3; \R^3),
  \end{align}
  with equality achieved for all $\m = \nabla v$ with
  $v \in \mathring{H}^1(\R^3)$.  Additionally, it is possible to prove
  that the spectrum of $\mathcal H$ is at most countable and contained
  in the interval $[0,1]$. Note that any element $\m \in \vDsol$, in
  particular, any configuration built as in Remark
  \ref{rmk:Dsolinfinitedim} belongs to the kernel of $\mathcal H$ (see
  \cite{friedman1981mathematical3} for a detailed analysis). Finally,
  we recall that $\mathcal H$ maps constant magnetizations in $\Omega$
  (and zero outside) into constant magnetic fields in $\Omega$ (but
  not constant outside) if and only if $\Omega$ is an ellipsoid
  \cite{di1986bubble,karp1994newtonian,di2016newtonian}. Thus, if
  $\Omega$ is an ellipsoid, the restriction of $\mathcal H$ to
  three-dimensional constant vector fields in $\Omega$ defines a
  finite-dimensional linear operator (the so called demagnetizing
  tensor), whose eigenvalues (the so-called demagnetizing factors) are
  among the most important quantities in ferromagnetism
  \cite{osborn1945demagnetizing}.
\end{remark}

\section{Micromagnetics of curved thin shells}

We now illustrate the utility of the variational principles discussed
in section \ref{s:var} in the case of dimension reduction for thin
ferromagnetic shells. Previously such results have been established
under suitable technical assumptions on the geometry of the surface in
the case of thin layers \cite{carbou01}, and shells enclosing convex
bodies \cite{difratta16}. Here we use Theorem \ref{mainthmscvec} to
give an elementary proof of the dimension reduction via
$\Gamma$-convergence, which does not require convexity or other purely
technical assumptions on the shape of the shell.

Let $\Omega$ be a bounded domain in $\RR^3$. For any
$\m \in H^1( \Omega, \Stwo^2 )$, the micromagnetic energy functional
in \eqref{eq:EE} in the absence of crystalline anisotropy and the
applied magnetic field reads
\begin{equation}
  \mathcal{G}_{\Omega}( \m ) \assign \frac{1}{2} \int_{\Omega} \left( 
    | \grad \m |^2 - \tmmathbf h_{\m} \cdot \m \right),
\end{equation}
where $\tmmathbf h_{\m}$ is the solution of
\eqref{eq:MA1}--\eqref{eq:MA3} with $\m$ extended by zero outside
$\Omega$.  Taking into account Theorem \ref{mainthmscvec}, the
following equivalent expressions arise:
\begin{eqnarray}
  \mathcal{G}_{\Omega}( \m )
  & \eqs & \frac{1}{2} \int_{\Omega}
           \left| \grad \m \right|^2 + \min_{\tmmathbf{a} \in
           \mathring{H}^1( \RR^3; 
           \RR^3 )} \mathcal{V}( \m, \tmmathbf{a} ), \\ 
  \mathcal{G}_{\Omega}( \m )
  & \eqs & \frac{1}{2} \int_{\Omega}
           \left| \grad \m \right|^2 + \max_{u \in \mathring{H}^1( \RR^3 )}
           \mathcal{W}(\m, u ) . 
\end{eqnarray}
In particular, if we define
\begin{equation}
  \mathcal{G}_{\Omega}( \m, \tmmathbf{a} ) \assign \frac{1}{2}
  \int_{\Omega} \left| \grad \m \right|^2 +\mathcal{V}( \m, 
  \tmmathbf{a} ), \quad \mathcal{G}_{\Omega}( \m, u )
  \assign \frac{1}{2} \int_{\Omega} \left| \grad \m \right|^2 +\mathcal{W}
  ( \m, u )
\end{equation}
then
\begin{eqnarray}
  \min_{\m \in H^1( \Omega, \Stwo^2 )} \mathcal{G}_{\Omega}(
  \m ) & \eqs & \min_{\m \in H^1( \Omega, \Stwo^2
  )} \min_{\tmmathbf{a} \in \mathring{H}^1( \RR^3; \RR^3 )}
  \mathcal{G}_{\Omega}( \m, \tmmathbf{a} ), \\
  \min_{\m \in H^1( \Omega, \Stwo^2 )} \mathcal{G}_{\Omega}(
  \m ) & \eqs & \min_{\m \in H^1( \Omega, \Stwo^2 )} \max_{u
  \in \mathring{H}^1( \RR^3 )} \mathcal{G}_{\Omega}( \m, u) . 
\end{eqnarray}
Thus, the minimization problem for the micromagnetic energy functional
$H^1 ( \Omega, \Stwo^2 )$ can be restated as a minimization problem on
the product space
$H^1( \Omega, \Stwo^2 ) \times \mathring{H}^1( \RR^3; \RR^3 )$, or as
a minimax problem on the spaces
$H^1( \Omega, \Stwo^2 ) \times \mathring{H}^1( \RR^3 )$.

Let $S$ be a compact $C^2$ surface in $\RR^3$. It is
well-known that $S$ is orientable and admits a tubular neighborhood
({\tmabbr{cf.}}~{\cite[Prop.~1, p.~113]{docarmo}}). Precisely, let
$\tmmathbf{n}: S \rightarrow \Stwo^2$ be the unit normal vector field
associated with the choice of an orientation of $S$. For every
$\xi \in S$, $\delta \in \RR_+$, denote by
$\ell_{\delta} (\xi) \assign \{ \xi + t\tmmathbf{n} (\xi) \}_{| t | <
  \delta}$ the normal segment to $S$ having radius $\delta$ and
centered at $\xi$. Then, there exists $\delta \in \RR_+$ such that the
following properties hold
({\tmabbr{cf.}}~{\cite[{\tmabbr{p.}}~112]{docarmo}}):
\begin{itemize}
  \item For every $\xi_1, \xi_2 \in S$ one has $\ell_{\delta} (\xi_1) \cap
  \ell_{\delta} (\xi_2) = \emptyset$ whenever $\xi_1 \neq \xi_2$.
  
  \item The union $\Omega_{\delta} \assign \cup_{\xi \in S} \ell_{\delta}
  (\xi)$ is an open set of $\RR^3$ containing $S$.
  
\item For $I:=(-1,1)$, set $\mathcal{M} \assign S \times I$. For every
  $\varepsilon \in I_{\delta}^+ \assign (0, \delta)$, the map
  \begin{equation}
    \psi_{\varepsilon} : (\xi, t) \in \mathcal{M} \mapsto \xi + \varepsilon
    t\tmmathbf{n} (\xi) \in \Omega_{\varepsilon} \label{eq:diffeomorphism}
  \end{equation}
  is a $C^1$ diffeomorphism of the product manifold $\mathcal{M}$ onto
  $\Omega_{\varepsilon}$. In particular, the nearest point projection
  $\pi : \Omega_{\varepsilon} \rightarrow S$, which maps any
  $x \in \Omega_{\varepsilon}$ onto the unique $\xi \in S$ such that
  $x \in \ell_{\varepsilon} (\xi)$, is a $C^1$ map. All integrals over
  $\mathcal M$ are with respect to the measure
  $\mathcal H^2 \times \mathcal L^1$.
\end{itemize}
The open set $\Omega_{\delta}$ is then called a tubular neighborhood
of $S$ of radius $\delta$. Note that
$\Omega_{\delta} \equiv \psi_{\delta} (\mathcal{M})$.  

In what follows, the symbols $\tau_1 (\xi), \tau_2 (\xi)$ denote the
orthonormal basis of $T_{\xi} S$ made by the principal directions at $\xi \in
S$. Also, we denote by $\sqrt{\mathfrak{g}_{\epsilon}}$ the metric factor
which relates the volume form on $\Omega_{\varepsilon}$ to the volume form on
$\mathcal{M}$, and by $\mathfrak{h}_{1, \varepsilon}, \mathfrak{h}_{2,
\varepsilon}$ the metric coefficients which transform the gradient on
$\Omega_{\varepsilon}$ into the gradient on $\mathcal{M}$. A direct
computation shows that
\begin{equation}\label{eq:hg}
  \sqrt{\mathfrak{g}_{\varepsilon} (\xi, t)} \assign | 1 + 2 \varepsilon t H
  (\xi) + \varepsilon^2 t^2 G (\xi) |, \quad \mathfrak{h}_{i, \varepsilon}
  (\xi, t) \assign (1 + \varepsilon t \kappa_i (\xi))^{- 1} \quad( i \in
  \NN_2 ),
\end{equation}
where $H (\xi)$ and $G (\xi)$ are, respectively, the mean and Gaussian
curvature at $\xi \in S$, and $\kappa_1 (\xi), \kappa_2 (\xi)$ are the
principal curvatures at $\xi \in S$. In what follows we always assume
the thickness $\delta$ to be sufficiently small so that the quantities
in \eqref{eq:hg} are uniformely bounded from both above and below by
some positive constants depending only on $S$.

We denote by $H^1( \mathcal{M}; \RR^3 )$ the Sobolev space of
vector-valued functions defined on $\mathcal{M}$ endowed with the norm $\|
\tmmathbf{m} \|^2_{H^1 (\mathcal{M})} \assign \| \tmmathbf{m} \|^2_{L^2
(\mathcal{M})} + \| \nabla_{\xi} \tmmathbf{m} \|^2_{L^2 (\mathcal{M})} + \|
\partial_t \tmmathbf{m} \|^2_{L^2 (\mathcal{M})}$ where $\nabla_{\xi}
\tmmathbf{m}$ stands for the tangential gradient of $\tmmathbf{m}$ on
$S$.
Finally, we write $H^1( \mathcal{M}; \Stwo^2 )$ for the subset of
$H^1( \mathcal{M}; \RR^3 )$ consisting of functions taking values
in $\Stwo^2$.

Next, for every $\varepsilon \in I_{\delta}^+$ we consider the
micromagnetic energy functional on
$H^1( \Omega_{\varepsilon}, \Stwo^2 )$ which, after normalization,
reads
\begin{eqnarray}
  \mathcal{G}_{\varepsilon} (\widetilde{\tmmathbf{m}}) & \assign & \frac{1}{2
  \varepsilon} \int_{\Omega_{\varepsilon}} \left| \grad \widetilde{\tmmathbf{m}}
  \right|^2 + \frac{1}{2 \varepsilon} \int_{\RR^3} \left| \grad
  u_{\widetilde{\tmmathbf{m}}} \right|^2, 
\end{eqnarray}
with $u_{\widetilde{\tmmathbf{m}}}$ being the unique solution in $\Wonegradh$ of
the Poisson equation $\lapl u_{\widetilde{\tmmathbf{m}}} = \divv 
\widetilde{\tmmathbf{m}}$, with the understanding that $\widetilde{\tmmathbf{m}}$ is
extended by zero outside of $\Omega_{\varepsilon}$. The change of variables
{\eqref{eq:diffeomorphism}} allows for the following equivalent expression of
the micromagnetic energy functional
\begin{equation}
  \mathcal{F}_{\varepsilon} (\tmmathbf{m}) \assign \mathcal{E}_{\varepsilon}
  (\tmmathbf{m}) + \frac{1}{2 \varepsilon} \int_{\RR^3} \left| \grad
  u_{\widetilde{\tmmathbf{m}}} \right|^2,
\end{equation}
with $\tmmathbf{m} (\xi, t) \assign \widetilde{\tmmathbf{m}} \circ
\psi_{\varepsilon} (\xi, t) \in H^1( \mathcal{M}; \Stwo^2 )$ for
$\widetilde{\tmmathbf{m}} \in H^1( \Omega_{\varepsilon}, \Stwo^2 )$,
and $\mathcal{E}_{\varepsilon}$ the family of Dirichlet energies on
$\mathcal{M}$ defined by
\begin{equation}
  \mathcal{E}_{\varepsilon} (\tmmathbf{m}) \assign \frac{1}{2}
  \int_{\mathcal{M}} \sum_{i \in \NN_2} | \mathfrak{h}_{i, \varepsilon}
  \partial_{\tau_i (\xi)} \tmmathbf{m} |^2  \sqrt{\mathfrak{g}_{\varepsilon}}
  + \frac{1}{2 \varepsilon^2} \int_{\mathcal{M}} | \partial_t \tmmathbf{m} |^2
  \sqrt{\mathfrak{g}_{\varepsilon}} . \label{eq:exenergy}
\end{equation}
We are interested in the limiting behavior of the minimizers of
$\mathcal{F}_{\varepsilon}$ when $\varepsilon \rightarrow 0$. In that regard,
we prove the following $\Gamma$-convergence result. 

\begin{theorem}\label{th:thinfilm}
  As $\varepsilon \to 0$, the following statements hold:
\begin{enumerate}
\item If the sequence
  $(\m_\varepsilon) \subset H^1(\mathcal M; \Stwo^2)$ satisfies
  $\mathcal{F}_{\varepsilon} (\m_\varepsilon) \leqslant C$, then upon
  possible extraction of a subsequence there exists
  $\m_0 \in H^1(\mathcal M; \Stwo^2)$ such that
  $\m_\varepsilon \rightharpoonup \m_0$ weakly in
  $H^1(\mathcal M; \Stwo^2)$.
\item The family
  $(\mathcal{F}_{\varepsilon})_{\varepsilon \in I_{\delta}^+}$ is
  equi-coercive in the weak topology of $H^1( \mathcal{M}; \Stwo^2 )$,
  and $(\mathcal{F}_{\varepsilon})_{\varepsilon \in I_{\delta}^+}$
  $\Gamma$-converges in that topology to the functional
  \begin{equation}
    \mathcal{F}( \m ) =  \begin{cases} \displaystyle
      \frac12 \int_{\mathcal M} \left[  | \nabla_{\xi}
        \tmmathbf{m} |^2 + (\tmmathbf{m} \cdot 
        \tmmathbf{n})^2 \right] \mathd \xi & \text{{\tmem{if}} }
      \partial_t \tmmathbf{m}= 0,\\
      + \infty & \text{{\tmem{otherwise}}}.
    \end{cases} 
  \end{equation}
\item If $\m_\varepsilon$ are minimizers of
  $\mathcal{F}_{\varepsilon}$, then upon possible extraction of a
  subsequence $(\m_\varepsilon)$ converges strongly in
  $H^1( \mathcal{M}; \Stwo^2 )$ to a minimizer of $\mathcal{F}$.
  \end{enumerate}
\end{theorem}

\begin{proof}
  The first statement is a direct consequence of the boundedness of
  the Dirichlet energy of $(\m)_{\varepsilon \in I_{\delta}^+}$.  The
  equi-corecivity of the family
  $(\mathcal{F}_{\varepsilon})_{\varepsilon \in I_{\delta}^+}$ is
  proved in {\cite{difratta16}}, where it is also proved the
  $\Gamma$-convergence of the Dirichlet energies
  $\mathcal{E}_{\varepsilon}$ to the energy functional
  \begin{equation}
    \mathcal{E}_0 : \m \in H^1( \mathcal{M}; \Stwo^2 ) \mapsto
    \begin{cases} \displaystyle
      \frac12 \int_{\mathcal M} | \nabla_{\xi} \m |^2 \mathd \xi
      \quad & \text{if } \partial_t \m 
      \eqs 0,\\
      + \infty & \text{otherwise} .
    \end{cases} \label{eq:GammaliEprime0}
  \end{equation}
  In particular, if $\tmmathbf{m} \in H^1( \mathcal{M}; \Stwo^2 )$,
  $\m(\xi, \cdot)$ is not constant for a.e.  $\xi \in S$, and
  $\tmmathbf{m}_{\varepsilon} \rightharpoonup \tmmathbf{m}$ weakly in
  $H^1 ( \mathcal{M}; \Stwo^2 )$, then necessarily
  $\limsup_{\varepsilon \rightarrow 0} \mathcal{F}_{\varepsilon}(
  \m_{\varepsilon} ) = + \infty$. Therefore, without loss of
  generality, we can restrict our analysis to families
  $(\tmmathbf{m}_{\varepsilon})_{\varepsilon \in I_{\delta}^+}$ in
  $H^1( \mathcal{M}; \Stwo^2 )$ such that
  $\tmmathbf{m}_{\varepsilon} (\xi, s) \rightharpoonup \m_0 (\xi)
  \chi_I (s)$ for some $\m_0 \in H^1( S, \Stwo^2 )$.
    \begin{figure}[ht]
    \centering
      \includegraphics{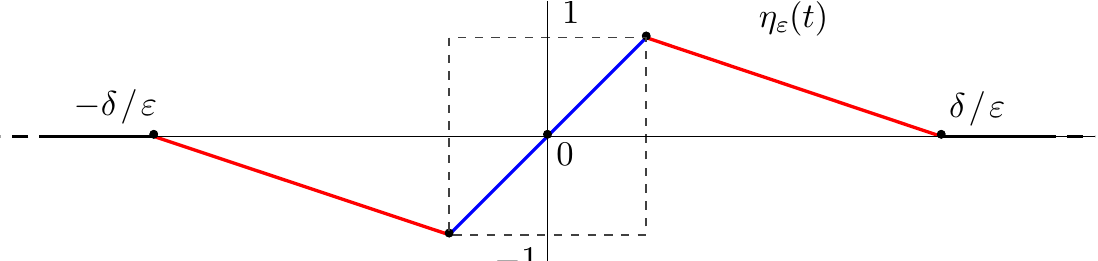}
      \caption{The function $\eta_{\varepsilon}$ used in the construction of
      the family of potentials.\label{Figure1}}
    \end{figure}
    
    \medskip
   \noindent 
   {\tmstrong{Step 1. $\Gamma$-$\mathrm{liminf}$ inequality.}} To
   shorten notation, it is convenient to introduce the
   $\nabla_{\varepsilon} \assign (\mathfrak{h}_{1, \varepsilon}
   \partial_{\tau_1 (\xi)}, \mathfrak{h}_{2, \varepsilon}
   \partial_{\tau_2 (\xi)}, \varepsilon^{-1} \partial_t)$. Then, to
   every
   $\widetilde{\tmmathbf{m}_\varepsilon} \in H^1(
   \Omega_{\varepsilon}, \Stwo^2 )$, $\tilde{u} \in \Wonegradh$, we
   associate the vector field
   $\m_\varepsilon \assign \widetilde{\tmmathbf{m}_\varepsilon} \circ
   \psi_{\varepsilon}$ and the scalar potential
   $u_\varepsilon \assign \tilde{u} \circ \psi_{\varepsilon}$.

   We use the characterization of the magnetostatic sef-energy given
   in Theorem~\ref{mainthmscvec}
   ({\tmabbr{cf.}}~{\eqref{eq:minscpotthm}}). For every $\delta > 0$,
   we denote by $\mathcal{M}_{\delta}$ the product manifold
   $\mathcal{M}_{\delta} \assign S \times I_{\delta}$. We have, with
   the identification of $H^1_0 (\Omega_{\delta})$ as a subspace of
   $\Wonegradh$:
  \begin{eqnarray}
    \frac{1}{2 \varepsilon} \int_{\RR^3} | \grad
    u_{\widetilde{\tmmathbf{m}}_{\varepsilon}} |^2 & \eqs & \max_{\tilde{u}
    \in \Wonegradh} \frac{1}{\varepsilon}
        \left( \int_{\Omega_{\varepsilon}}
    \grad \tilde{u} \cdot \widetilde{\tmmathbf{m}}_{\varepsilon} - \frac{1}{2}
    \int_{\RR^3} \left| \grad \tilde{u} \right|^2 \right) \notag \\
    & \geqslant & \max_{\tilde{u} \in H^1_0 (\Omega_{\delta})}
    \frac{1}{\varepsilon}
    \left( \int_{\Omega_{\varepsilon}} \grad \tilde{u}
    \cdot \widetilde{\tmmathbf{m}}_{\varepsilon} - \frac{1}{2}
    \int_{\Omega_{\delta}} \left| \grad \tilde{u} \right|^2 \right) \notag \\
    & \eqs & \max_{\tilde{u} \in H^1_0 (\Omega_{\delta})}
    \left(
    \int_{\mathcal{M}} \grad_{\varepsilon} [\tilde{u} \circ
    \psi_{\varepsilon}] \cdot \m_{\varepsilon} 
    \sqrt{\mathfrak{g}_{\varepsilon}} - \frac{1}{2} \int_{\mathcal{M}_{\delta
    / \varepsilon}} \left| \grad_{\varepsilon} [\tilde{u} \circ
    \psi_{\varepsilon}] \right|^2 \sqrt{\mathfrak{g}_{\varepsilon}} \right) \notag \\
    & \geqslant & \int_{\mathcal{M}} \grad_{\varepsilon}
    u_{\varepsilon} \cdot \m_{\varepsilon} 
    \sqrt{\mathfrak{g}_{\varepsilon}} - \frac{1}{2}
    \int_{\mathcal{M}_{\delta / \varepsilon}} \left| \grad_{\varepsilon}
    u_{\varepsilon} \right|^2 \sqrt{\mathfrak{g}_{\varepsilon}}, 
    \label{eq:liminfpotmag}
  \end{eqnarray}
  for every $u_{\varepsilon} = \tilde{u} \circ \psi_{\varepsilon}$
  with $\tilde{u} \in H^1_0 (\Omega_{\delta})$. Note that
  $u_\varepsilon$ is well defined on
  $\mathcal M_{\delta/\varepsilon}$. Next, we build the family of
  potentials ({\tmabbr{cf.}}~Figure~\ref{Figure1})
  \begin{equation}
    u_{\varepsilon} (\xi, t) \assign \varepsilon
    \eta_{\varepsilon} (t) ( \m_0 (\xi) 
    \cdot \tmmathbf{n} 
    (\xi) ), \qquad \eta_\varepsilon (t) \assign \left\{
    \begin{array}{ll}
      t & \text{if } | t | < 1,\\
      \frac{\delta - \varepsilon | t |}{
      \delta - \varepsilon}
        & \text{if
          } 1 \leqslant | t | < \delta / \varepsilon,\\
      0 & \text{if } | t | > \delta / \varepsilon .
    \end{array} \right. \label{eq:eta}
  \end{equation}
  Note that $\eta_{\varepsilon} (t) = 0$ if $| t | > \delta / \varepsilon >
  1$. Also we have
  \begin{equation}
    \eta_{\varepsilon}' (t) = 1 \quad \mathrm{if \ } | t | < 1, \qquad
    (\eta'_{\varepsilon} (t))^2 = \frac{\varepsilon^2}{(\delta -
    \varepsilon)^2} \quad \mathrm{if \ } 1 < | t | < \delta / \varepsilon .
  \end{equation}
  Hence, we have
  $\grad_{\xi} u_{\varepsilon} (\xi, t) = \varepsilon
  \eta_{\varepsilon} (t) \grad_{\xi}( \m_0 (\xi) \cdot \tmmathbf{n}
  (\xi) )$ and
  $\partial_t u_{\varepsilon} (\xi, t) = \varepsilon
  \eta'_{\varepsilon} (t) ( \m_0 (\xi) \cdot \tmmathbf{n} (\xi) )$. It
  follows that
  $\left\| \grad_{\xi} u_{\varepsilon}
  \right\|^2_{\mathcal{M}_{_{\delta / \varepsilon}}} \rightarrow 0$ as
  $\varepsilon \rightarrow 0$.  Therefore, from
  {\eqref{eq:liminfpotmag}} and \eqref{eq:eta} we obtain
  \begin{equation}
    \liminf_{\varepsilon \rightarrow 0} \frac{1}{2 \varepsilon} \int_{\RR^3}
    | \grad u_{\widetilde{\tmmathbf{m}}_{\varepsilon}} |^2 \; \geqslant
    \; \int_{\mathcal{M}}( \m_0 \cdot \tmmathbf{n} )^2 -
    \frac{1}{2} \limsup_{\varepsilon \rightarrow 0} \int_{\mathcal{M}_{\delta /
        \varepsilon}}( \m_0 \cdot \tmmathbf{n} )^2
    (\eta'_{\varepsilon} (t))^2 \mathd \xi \mathd t.
  \end{equation}
  On the other hand, we have 
  \begin{eqnarray}
    \int_{\mathcal{M}_{\delta / \varepsilon}}( \m_0 (\xi) \cdot \tmmathbf{n}
    )^2 (\eta'_{\varepsilon} (t))^2 \mathd \xi \mathd t & \eqs & \left(
    1 + \frac{\varepsilon}{\delta - \varepsilon} \right) \int_{\mathcal{M}}
   ( \m_0 \cdot \tmmathbf{n} )^2 \xrightarrow{\varepsilon
    \rightarrow 0} \int_{\mathcal{M}}( \m_0 \cdot \tmmathbf{n} )^2
    . 
  \end{eqnarray}
  Summarizing, we get 
  \begin{equation}
    \liminf_{\varepsilon \rightarrow 0} \frac{1}{2 \varepsilon} \int_{\RR^3}
    | \grad u_{\widetilde{\tmmathbf{m}}_{\varepsilon}} |^2 \; \;
    \geqslant \; \frac{1}{2} \int_{\mathcal{M}}( \m_0 \cdot \tmmathbf{n}
    )^2  \eqs \int_S( \m_0 \cdot \tmmathbf{n} )^2 .
    \label{eq:temp3}
  \end{equation}
  Taking into account {\eqref{eq:GammaliEprime0}}, we conclude that for any
  $(\tmmathbf{m}_{\varepsilon})_{\varepsilon \in I_{\delta}}$ in $H^1(
  \mathcal{M}; \Stwo^2 )$ such that $\tmmathbf{m}_{\varepsilon} (\xi, s)
  \rightharpoonup \m_0 (\xi) \chi_I (s)$ for some $\m_0 \in H^1( S,
  \Stwo^2 )$, the following lower bound holds
  \begin{eqnarray}
    \liminf_{\varepsilon \rightarrow 0} \mathcal{F}_{\varepsilon}(
    \m_{\varepsilon} ) & \geqslant & \frac{1}{2} \int_{\mathcal{M}}
    \left| \grad_{\xi}  \m_0 \right|^2 + \frac{1}{2} \int_{\mathcal{M}}(
    \m_0 \cdot \tmmathbf{n} )^2 .  \label{eq:minprobauxsp}
  \end{eqnarray}

\medskip
   \noindent 
    {\tmstrong{Step 2. Recovery sequence.}} We now show that, for any
    $\tmmathbf{m}_0 \in H^1( S, \Stwo^2 )$, the constant family of
    magnetizations given by $\tmmathbf{m}_{\varepsilon} (\xi, t) \assign
    \tmmathbf{m}_0 (\xi) \chi_I (t)$ defines a recovery sequence. It is clear
    that such a family of functions works for the exchange energies
    $\mathcal{E}_{\varepsilon}$ due to {\eqref{eq:GammaliEprime0}}. Therefore,
    we can focus on the magnetostatic self-energy. To shorten notation, it is
    convenient to introduce the symbol $\curl_{\varepsilon}
    \tmmathbf{a}^{\star} \assign \curl_{\varepsilon, \xi} \tmmathbf{a}^{\star}
    + \varepsilon^{- 1} \tmmathbf{n} \times \partial_t \tmmathbf{a}^{\star}$
    with
  \begin{equation}
    \curl_{\varepsilon, \xi} \tmmathbf{a}^{\star} = \sum_{i = 1}^2
    \mathfrak{h}_{i, \varepsilon} (\xi, t) \left(\tmmathbf{\tau}_i (\xi) \times
    \partial_{\tau_i (\xi)} \tmmathbf{a}^{\star}\right) .
  \end{equation}
  By the expression of the magnetostatic self-energy in terms of
  magnetic vector potential given in Theorem~\ref{mainthmscvec}
  ({\tmabbr{cf.}}~{\eqref{eq:minvecpotnodiv}}), we have
  \begin{eqnarray}
    \frac{1}{2 \varepsilon} \int_{\RR^3} \left| \grad
    u_{\widetilde{\tmmathbf{m}}_{\varepsilon}} \right|^2
    & \eqs &
             \frac{1}{2 \varepsilon} \,
             \underset{\tilde{\tmmathbf{a}}^{\star} \in
             \mathring{H}^1( \RR^3; \RR^3 
             )}{\min} \int_{\RR^3} \left( \left| \grad
             \tilde{\tmmathbf{a}}^{\star} \right|^2 + |
             \widetilde{\tmmathbf{m}}_{\varepsilon} |^2 - 2 \curl 
             \tilde{\tmmathbf{a}}^{\star} \cdot \widetilde{\tmmathbf{m}}_{\varepsilon} 
             \right) \notag \\
    & \leqslant & \underset{\tilde{\tmmathbf{a}}^{\star} \in {H}^1_0
                  (\Omega_{\delta},\RR^3)}{\min} \left( \frac{1}{2} \int_{\mathcal{M}_{\delta /
                  \varepsilon}} \left| \grad_{\varepsilon}  [\tilde{\tmmathbf{a}}^{\star}
                  \circ \psi_{\varepsilon}] \right|^2 \sqrt{\mathfrak{g}_{\varepsilon}} \right. +
                  \nonumber\\
    &  & \quad \quad \quad \quad \quad \quad \quad + \left. \frac{1}{2}
         \int_{\mathcal{M}}( \left| \m_{\varepsilon} \right|^2 - 2
         \curl_{\varepsilon}  [\tilde{\tmmathbf{a}}^{\star} \circ
         \psi_{\varepsilon}] \cdot \m_{\varepsilon} )
         \sqrt{\mathfrak{g}_{\varepsilon}}  \right) \notag \\
    & \leqslant & \frac{1}{2} \int_{\mathcal{M}} \left( \left|
                  \m_{\varepsilon} \right|^2 - 2 \curl_{\varepsilon}
                  \tmmathbf{a}^{\star} \cdot \m_{\varepsilon} \right)
                  \sqrt{\mathfrak{g}_{\varepsilon}} + \frac{1}{2} \int_{\mathcal{M}_{\delta
                  / \varepsilon}} \left| \grad_{\varepsilon}
                  \tmmathbf{a}^{\star} \right|^2 \sqrt{\mathfrak{g}_{\varepsilon}}, 
  \end{eqnarray}
  for every $\tmmathbf{a}^{\star} = \tilde{\tmmathbf{a}}^{\star} \circ
  \psi_{\varepsilon}$ with $\tilde{\tmmathbf{a}}^{\star} \in {H}^1_0
  (\Omega_{\delta},\RR^3)$. Next, we consider the family of potentials
  \begin{equation}
    \tmmathbf{a}^{\star}_{\varepsilon} (\xi, t) \assign
    \varepsilon \eta_{\varepsilon} (t) ( \m_0 (\xi)
    \times \tmmathbf{n} (\xi) ),
  \end{equation}
  with $\eta_{\varepsilon}$ given by {\eqref{eq:eta}}. We get that
  \begin{equation}
    \grad_{\xi} \tmmathbf{a}^{\star}_{\varepsilon} (\xi, t) =
    \varepsilon \eta_{\varepsilon} (t)  \grad_{\xi}
    ( \m_0 (\xi) \times \tmmathbf{n} (\xi) ), \quad \partial_t
    \tmmathbf{a}^{\star}_{\varepsilon} (\xi, t) = \varepsilon
    \eta'_{\varepsilon} (t)  (\m_0 (\xi) \times \tmmathbf{n} (\xi)
    ) . 
  \end{equation}
  Hence, we have
  $\left\| \grad_{\xi} \tmmathbf{a}^{\star}_{\varepsilon}
  \right\|^2_{\mathcal{M}_{_{\delta / \varepsilon}}} \rightarrow 0$ as
  $\varepsilon \rightarrow 0$. Therefore
  \begin{eqnarray}
    \limsup_{\varepsilon \rightarrow 0} \frac{1}{2 \varepsilon} \int_{\RR^3}
    \left| \grad u_{\widetilde{\tmmathbf{m}}_{\varepsilon}} \right|^2
    & \leqslant
    & \frac{1}{2} \int_{\mathcal{M}} \left| \m_0 \right|^2 -
      \int_{\mathcal{M}} \left[ 
      \tmmathbf{n} \times( \m_0 \times \tmmathbf{n} ) \right] \cdot
      \m_0 \notag \\
    &  & \quad \hspace{3em} \quad + \limsup_{\varepsilon \rightarrow 0}
         \left( \frac{1}{2} \int_{\mathcal{M}_{\delta / \varepsilon}} \left|(
         \m_0 \times \tmmathbf{n} ) \eta_{\varepsilon}' (t) \right|^2 \mathd
         t \,\right) . 
  \end{eqnarray}
  Moreover, we have
  \begin{eqnarray}
    \int_{\mathcal{M}_{\delta / \varepsilon}} \left|( \m_0 \times
    \tmmathbf{n} ) \eta_{\varepsilon}' (t) \right|^2 \mathd t
    & \eqs &
             \left( 1 + \frac{\varepsilon}{\delta - \varepsilon} \right)
             \int_{\mathcal{M}} \left| \m_0 \times \tmmathbf{n} \right|^2 \;
             \xrightarrow{\varepsilon \rightarrow 0} \int_{\mathcal{M}} \left| \m_0
             \times \tmmathbf{n} \right|^2 . 
  \end{eqnarray}
  Summarizing, we get 
  \begin{eqnarray}
    \limsup_{\varepsilon \rightarrow 0} \frac{1}{2 \varepsilon} \int_{\RR^3}
    \left| \grad u_{\widetilde{\tmmathbf{m}}_{\varepsilon}} \right|^2
    & \leqslant
    & \frac{1}{2} \int_{\mathcal{M}} \left| \m_0 \right|^2 - \int_{\mathcal{M}}  \left[
      \tmmathbf{n} \times( \m_0 \times \tmmathbf{n} ) \right] \cdot
      \m_0 + \frac{1}{2} \int_{\mathcal{M}} \left| \m_0 \times \tmmathbf{n}
      \right|^2 \notag \\
    & \eqs & \frac{1}{2} \int_{\mathcal{M}}( \m_0 \cdot \tmmathbf{n}
             )^2 . 
  \end{eqnarray}
  Strong convergence of minimizers $\m_\varepsilon \to \m_0$ in
  $H^1(\mathcal M; \Stwo^2)$ follows from weak convergence in
  $H^1(\mathcal M; \Stwo^2)$ and convergence of the norms
\begin{equation}
  \int_{\mathcal{M}} \sum_{i \in \NN_2} |  \partial_{\tau_i (\xi)}
  \tmmathbf{m}_\varepsilon |^2  
  +  \int_{\mathcal{M}} | \partial_t \tmmathbf{m}_\varepsilon |^2 \to
  \int_{\mathcal M} | \nabla_{\xi} \m_0 |^2,
\end{equation}
where the latter is a straightforward consequence of
$\mathcal E_\varepsilon (\m_\varepsilon) \to \mathcal E_0(\m)$ for a
minimizing sequence $(\m_\varepsilon)$.  This completes the proof.
\end{proof}

\section*{Acknowledgements}

G.\,D.\,F. acknowledges support from the Austrian Science Fund (FWF)
through the special research program ``Taming complexity in partial
differential systems'' (Grant SFB F65) and of the Vienna Science and
Technology Fund (WWTF) through the research project ``Thermally
controlled magnetization dynamics'' (Grant MA14-44). C.\,B.\,M. was
supported, in part, by NSF via grants DMS-1614948 and DMS-1908709. The
work of F.\,N.\,R. was supported by the Swedish Research Council Grant
No. 642-2013-7837 and by G\"{o}ran Gustafsson Foundation for Research
in Natural Sciences and Medicine. V.\,V.\,S. acknowledges support from
Leverhulme grant RPG-2018-438. G.\,D.\,F., C.\,B.\,M. and
V.\,V.\,S. would also like to thank the Isaac Newton Institute for
Mathematical Sciences for support and hospitality during the program
``The mathematical design of new materials'' supported by EPSRC via
grants EP/K032208/1 and EP/R014604/1.

\bibliographystyle{siam}
\bibliography{nonlin.bib,mura.bib,FNR.bib}

\end{document}